\documentclass{article}

\usepackage[round]{natbib}
\usepackage{amsmath,amssymb,amsthm,bm,enumerate,mathrsfs,mathtools}
\usepackage{latexsym,color,verbatim,multirow}
\usepackage[margin=1.3in]{geometry}
\usepackage{graphicx}
\usepackage{caption}
\usepackage{subcaption}
\usepackage{tikz}
\usepackage{authblk}
\usetikzlibrary{shapes,arrows}
\tikzstyle{block} = [rectangle, draw, fill=white!20,
    text width=7em, text centered, rounded corners, minimum height=4em]
\tikzstyle{title} = [text width=7em, text centered, font=\bfseries]
\tikzstyle{line} = [draw, -latex']

\usepackage{mycommands}

\newtheorem{theorem}{Theorem}
\newtheorem{corollary}[theorem]{Corollary}

\newtheorem{proposition}[theorem]{Proposition}

\newtheorem{model}[theorem]{Model}

\theoremstyle{definition}
\newtheorem{example}{Example}

\newcommand{\cQ}{\mathcal{Q}}
\newcommand{\cY}{\mathcal{Y}}

\newcommand{\cV}{\mathcal{V}}
\newcommand{\bX}{X}

\newcommand{\hcQ}{\widehat{\mathcal{Q}}}

\newcommand{\Gv}{\;\;\big|\;\;}
\newcommand{\proj}{\cP}
\newcommand{\pow}{\text{Pow}}
\newcommand{\sF}{\mathscr{F}}
\newcommand{\msF}{m\mathscr{F}}
\newcommand{\sC}{\mathscr{C}}
\newcommand{\hM}{\widehat{M}}
\newcommand{\hI}{\widehat{I}}
\newcommand{\cI}{\mathcal{I}}
\newcommand{\leqAS}{\overset{\textrm{a.s.}}{\leq}}
\newcommand{\cL}{\mathcal{L}}

\newcommand*\mystrut{\vrule width0pt height0pt depth1.5ex\relax}
\newcommand{\underlabel}{\underbracket[1pt][.5pt]{\mystrut \quad\;\; \sub \quad\;\; }}

\newcommand{\sampOrData}{data }
\newcommand{\capSampOrData}{Data }

\begin{document}

\title{Optimal Inference After Model Selection}
\author{William Fithian\footnote{To whom correspondence should be addressed}}
\affil{Department of Statistics, University of California Berkeley}
\author{Dennis L. Sun}
\affil{Department of Statistics, California Polytechnic State University}
\author{Jonathan Taylor}
\affil{Department of Statistics, Stanford University}
\maketitle

\begin{abstract}
  To perform inference after model selection, we propose controlling the {\em selective type I error}; i.e., the error rate of a test given that it was performed. By doing so, we recover long-run frequency properties among selected hypotheses analogous to those that apply in the classical (non-adaptive) context. Our proposal is closely related to \sampOrData splitting and has a similar intuitive justification, but is more powerful. Exploiting the classical theory of~\citet{lehmann1955completeness}, we derive most powerful unbiased selective tests and confidence intervals for inference in exponential family models after arbitrary selection procedures. For linear regression, we derive new selective $z$-tests that generalize recent proposals for inference after model selection and improve on their power, and new selective $t$-tests that do not require knowledge of the error variance.
\end{abstract}

\section{Introduction}\label{sec:intro}

\noindent A typical statistical investigation can be thought of as consisting of two stages:
\begin{description}
\item[1. Selection:] The analyst chooses a statistical model for the data at hand, and formulates testing, estimation, or other problems in terms of unknown aspects of that model.
\item[2. Inference:] The analyst investigates the chosen problems using the data and the selected model.
\end{description}

Informally, the selection stage determines what questions to ask, and the inference stage answers those questions. Most statistical methods carry an implicit assumption that selection is {\em non-adaptive} --- that is, choices about which model to use, hypothesis to test, or parameter to estimate, are made before seeing the data. {\em Adaptive selection} (also known colloquially as ``data snooping'') violates this assumption, formally invalidating any subsequent inference.

In some cases, it is possible to specify the question prior to collecting the data---for instance, if the data are governed by some known physical law. However, in most applications, the choice of question is at least partially guided by the data. For example, we often perform exploratory analyses to decide which predictors or interactions to include in a regression model or to check whether the assumptions of a test are satisfied. The goal of this paper is to codify what it means for inference to be valid in the presence of adaptive selection and to propose methods that achieve this ``selective validity.''

If we do not account properly for adaptive model selection, the resulting inferences can have troubling frequency properties, as we now illustrate with an example.

\begin{example}[File Drawer Effect]
\label{ex:file_drawer}
Suppose one or more scientific research groups make $n$ independent measurements of $n$ quantities, ${Y_i \sim N(\mu_i, 1)}$. They focus only on the apparently large effects, selecting (say) only the indices $i$ for which $|Y_i|>1$, i.e.
\[ \hI = \{i : |Y_i| > 1\}. \]
Each scientist wishes to test $H_{0,i}: \mu_i = 0$ for his own $i \in \hI$ at significance level $\alpha = 0.05$. Most practitioners intuitively recognize that the nominal test that rejects $H_{0,i}$ when $|Y_i| > 1.96$ is invalidated by the selection.

What exactly is ``invalid'' about this test? After all, the probability of falsely rejecting a given $H_{0,i}$ is still $\P(|Y_i| > 1.96) = 0.05$, since $H_{0,i}$ is simply not tested at all most of the time. Rather, the troubling feature is that the error rate among the hypotheses {\em selected} for testing is possibly much higher than $\alpha$.
To be precise, let $n_0$ be the number of true null effects and suppose $n_0 \to \infty$ as $n\to\infty$. Then, in the long run, the fraction of errors among the true nulls we test is
\begin{align}
\frac{\text{\# false rejections}}{\text{\# true nulls selected}} &= \genfrac{}{}{}{}
{\raisebox{10pt}{$\dfrac{1}{n_0}\;\;\raisebox{3pt}{$\displaystyle\underset{{i:\,H_{0,i}\text{ true}}}\sum$} 1\{ i\in \hI,\ \text{reject $H_{0,i}$} \}$}}
{\raisebox{-7pt}{$\dfrac{1}{n_0}\;\;\raisebox{3pt}{$\displaystyle\underset{{i:\,H_{0,i}\text{ true}}}\sum$} 1\{ i \in \hI \}$}}
\nonumber\\[5pt]
&\to \frac{\P_{H_{0,i}}(i \in \hI,\ \text{reject $H_{0,i}$})}{\P_{H_{0,i}}(i \in \hI)} \nonumber\\
&= \P_{H_{0,i}}(\text{reject $H_{0,i}$}\ |\ i \in \hI), \label{eq:selective_type_1_error}
\end{align}
which for the nominal test is $\Phi(-1.96) / \Phi(-1)  \approx .16$.

Thus, we see that \eqref{eq:selective_type_1_error}, the probability of a false rejection conditional on selection, is a natural error criterion to control in the presence of selection. In this example, we can directly control \eqref{eq:selective_type_1_error} at level $\alpha=0.05$ simply by finding the critical value $c$ solving
\begin{align*}
\P_{H_{0,i}}\left( |Y_i| > c\ \big|\ |Y_i| > 1 \right) &= 0.05.
\end{align*}
In this case $c = 2.41$, which is more stringent than the nominal 1.96 cutoff.
\end{example}

This paper will develop a theory for inference after selection based on controlling the {\em selective type I error rate} \eqref{eq:selective_type_1_error}. Our guiding principle is:
\begin{center}
The answer must be valid, given that the question was asked.
\end{center}

For all its simplicity, Example~\ref{ex:file_drawer} can be regarded as a stylized model of science. Imagine that each $Y_i$ represents an estimated effect size from a scientific study. However, only the large estimates are ever published---a caricature which may not be too far from the truth, as recently demonstrated by \citet{franco14}. To compound the problem, there may be many reasonable methodologies to choose from, even once the analyst has decided roughly what scientific question to address \citep{gelman2013garden}. Because of the resulting selection bias, the error rate among published claims may be very high, leading even to speculation that ``most published research findings are false'' \citep{ioannidis2005most}. Thus, selection effects may be a partial explanation for the replicability crisis reported in the scientific community \citep{yong2012replication} and the popular media \citep{johnson2014new}.

The setting of Example~\ref{ex:file_drawer} has been studied extensively in the literature of simultaneous and selective inference, and several authors have proposed adjusting for selection by means of conditional inference. \citet{zollner2007overcoming} and \citet{zhong2008bias} construct selection-adjusted estimators and intervals for genome-wide association studies for genes that pass a fixed initial significance threshold, based on a conditional Gaussian likelihood.  \citet{cohen1989two} obtain unbiased estimates for the mean of the population whose sample mean is largest by conditioning on the ordering of the observed sample means, and \citet{sampson2005drop} and \citet{sill2009drop} apply the same idea to obtain estimates for the best-performing drug in an adaptive clinical trial design. \citet{hedges1984estimation} and \citet{hedges1992modeling} propose methods to adjust for the file drawer effect in meta-analysis when scientists only publish significant results.

Another framework for selection adjustment is proposed by \citet{benjamini2005false}, who consider the problem of constructing intervals for a number $R$ of parameters selected after viewing the data. Letting $V$ denote the number of non-covering intervals among those constructed, they define the {\em false coverage-statement rate} (FCR) as the expected fraction $V/\max(R,1)$ of non-covering intervals. Controlling the FCR at level $\alpha$ thus amounts to ``coverage on the average, among selected intervals.'' As we will see further in Section~\ref{sec:multiple}, FCR control is closely related to the selective error control criterion we propose. In fact, \citet{weinstein2013selection} employ conditional inference to construct FCR-controlling intervals in the context of Example~\ref{ex:file_drawer}. \citet{rosenblatt2014selective} propose a similar method for finding correlated regions of the brain, also with a view toward FCR control.

\subsection{Conditioning on Selection}\label{sec:conditioning}

In classical statistical inference, the notion of ``inference after selection'' does not exist. The analyst must specify the model, as well as the hypothesis to be tested, in advance of looking at the data. A classical level-$\alpha$ test for a hypothesis $H_0$ under model $M$ must control the usual or {\em nominal type I error rate}:
\begin{equation}
\P_{M, H_0}(\text{reject $H_0$}) \leq \alpha.
\label{eq:nominal_type_1}
\end{equation}
The subscript in~(\ref{eq:nominal_type_1}) reminds us that the probability is computed under the assumption that the data $Y$ are generated from model $M$, and $H_0$ is true; if $M$ is misspecified, there are no guarantees on the rejection probability.

In most statistical practice, it is unrealistic to rule out model selection altogether: statisticians are trained to check their models and to tweak them if they diagnose a problem (to a purist, even model checking is suspect, since it leaves open the possibility that the model will change after we see the data). We will argue that if the model and hypothesis are selected adaptively, we should instead control the selective type I error rate
\begin{equation}
\P_{M,H_0}( \text{reject $H_0$}\ |\ \text{$(M, H_0)$ selected} ) \leq \alpha.
\label{eq:selective_type_1}
\end{equation}

One can argue that models and hypotheses are practically never truly fixed but are chosen randomly, since they are based on the outcomes of previous experiments in the (random) scientific process. Typically, we ignore the random selection and use classical tests that control \eqref{eq:nominal_type_1}, implicitly assuming that the randomness in selecting $M$ and $H_0$ is independent of the data used for inference. In that case,
\begin{equation}
\P_{M, H_0}( \text{reject $H_0$} \ |\ \text{$(M, H_0)$ selected} ) = \P_{M, H_0}(\text{reject $H_0$}).
\label{eq:independence}
\end{equation}

While it may seem pedantic to point out that model selection is random if based on previous experiments, this viewpoint justifies a common prescription for what to do when previous experiments do not dictate a model. If it is possible to split the data $Y = (Y_{1}, Y_{2})$ with $Y_{1}$ independent of $Y_{2}$, then we can imitate the scientific process by setting aside $Y_{1}$ for selection and $Y_{2}$ for inference. If selection depends on $Y_{1}$ only, then any nominal level-$\alpha$ test based on the value of $Y_{2}$ will satisfy \eqref{eq:independence}, so the nominal test based on $Y_{2}$ also controls the selective error \eqref{eq:selective_type_1}.

This meta-algorithm for generating selective procedures from nominal ones is called {\em data splitting} or {\em sample splitting}. The idea dates back at least as far as \citet{cox1975note}, and, despite the paucity of literature on the topic, is common wisdom among practitioners. For example, it is customary in genetics to use one cohort to identify loci of interest and a separate cohort to confirm them \citep{sladek2007genome}. \citet{wasserman2009high} and \citet{meinshausen2009p} discuss data-splitting approaches to high-dimensional inference.

Data splitting owes much of its popularity to its transparent justification, which even a non-expert can appreciate: if we imagine that $Y_{1}$ is observed ``first,'' then we can proceed to analyze $Y_{2}$ as though model selection took place ``ahead of time.'' Equation (\ref{eq:independence}) guarantees that this temporal metaphor will not lead us astray even if it does not describe how $Y_{1}$ and $Y_{2}$ were actually collected.

\capSampOrData splitting elegantly solves the problem of controlling selective error, but at a cost. It not only reduces the amount of data available for inference, but also reduces the amount of data available for selection. Furthermore, it is not always possible to split the data into independent parts, as in the case of autocorrelated spatial and time series data.

In this article, we propose directly controlling the selective error rate \eqref{eq:selective_type_1} by conditioning on the event that $(M,H_0)$ is selected. As with \sampOrData splitting, we treat the data as though it were revealed in stages: in the first stage, we ``observe'' just enough data to resolve the decision of whether to test $(M,H_0)$, after which we can treat the data $(Y \gv (M, H_0)\text{ selected})$ as ``not yet observed'' when the second stage commences.

The intuition of the above paragraph can be expressed formally in terms of the filtration
\begin{equation}\label{eq:sampleCarving}
  \sF_0 \underlabel_{\text{used for selection}} \sF(\1_A(Y))
  \underlabel_{\text{used for inference}} \sF(Y),
\end{equation}
where $\sF(Z)$ denotes the $\sigma$-algebra generated by a random variable $Z$ (informally, everything we know about the data after observing $Z$), $\sF_0$ is the trivial $\sigma$-algebra (representing complete ignorance), and $A$ is the {\em selection event} $\{(M, H_0)\text{ selected}\}$. We can think of ``time'' as progressing from left to right in~(\ref{eq:sampleCarving}). In stage one, we learn just enough to decide whether to test $(M,H_0)$, and no more, advancing our state of knowledge from $\sF_0$ to $\sF(\1_A(Y))$. We then begin stage two, in which we discover the actual value of $Y$, advancing our knowledge to $\sF(Y)$. Because our selection decision is made at the end of stage one, everything revealed during stage two is fair game for inference.

In effect, controlling the type I error conditional on $A$ prevents us from appealing to the fact that $Y\in A$ as evidence against $H_0$. Even if $Y\in A$ is extremely surprising under $H_0$, we still will not reject unless we are surprised anew in the second stage.
In this sense, conditioning on a random variable discards the information it carries about any parameter or hypothesis of interest. In contrast to data splitting, which can be viewed as conditioning on $Y_{1}$ instead of $\1_A(Y_{1})$, we advocate discarding as little information as possible and reserving the rest for stage two. This frugality results in a more efficient division of the information carried by $Y$, which we call {\em \sampOrData carving}.

\subsection{Outline}

In Section~\ref{sec:selInf} we formalize the problem of selective inference, discuss general properties of selective error control, and address key conceptual questions. Conditioning on the selection event effectively discards the information used for selection, but some information is left over for second-stage inference. We will also see that a major advantage of selective error control is that it allows us to consider only one model at a time when designing tests and intervals, even if {\em a priori} there are many models under consideration.

If $\cL(Y)$, the law of random variable $Y$, follows an exponential family model, then for any event $A$, $\cL(Y \gv A)$ follows a closely related exponential family model. As a result, selective inference dovetails naturally with the classical optimality theory of \citet{lehmann1955completeness}; Section~\ref{sec:exfam} briefly reviews this theory and derives most powerful unbiased selective tests in arbitrary exponential family models after arbitrary model selection procedures. Because conditioning on more data than is necessary saps the power of second-stage tests, \sampOrData splitting yields inadmissible selective tests under general conditions.

Section~\ref{sec:approxInt} gives some general strategies for
computing rejection cutoffs for the tests prescribed in Section~\ref{sec:exfam}, while Sections~\ref{sec:linReg}--\ref{sec:examples} derive selective tests in specific examples. Section~\ref{sec:linReg} focuses on the case of linear regression, generalizing the recent proposals of~\citet{tibshirani2014exact}, \citet{lee2016exact}, and others. We
derive new, more powerful selective $z$-tests, as well as selective $t$-tests that do not require knowledge of the error variance $\sigma^2$.

Several simulations in Section~\ref{sec:simulation} compare the post-lasso selective $z$-test with \sampOrData splitting, and illustrate a {\em selection--inference tradeoff}, between using more data in the initial stage and reserving more information for the second stage. Section~\ref{sec:multiple} compares and contrasts selective inference with multiple inference, and Section~\ref{sec:discussion} concludes.

\section{The Problem of Selective Inference}\label{sec:selInf}

\subsection{Example: Regression and the Lasso}\label{sec:lassoRegression}

In the previous section, we motivated the idea of conditioning on selection. Arguably, the most familiar example of this ``selection'' is variable selection in linear regression. In regression, the observed data $Y \in \R^n$ is assumed to be generated from a multivariate normal distribution
\begin{equation}\label{eq:fullModel}
  Y \sim N_n(\mu, \;\sigma^2 I_n).
\end{equation}
The goal is to model the mean $\mu$ as a linear function of predictors $X_j$, $j=1,\ldots, p$. To obtain a more parsimonious model (or simply an identifiable model when $p > n$), researchers will often use only a subset $M \sub\{1,\ldots, p \}$ of the predictors. Each subset $M$ leads to a different statistical model corresponding to the assumption $\mu=\bX_M \beta^{M}$, where $\bX_M$ denotes the matrix consisting of columns $X_j$ for $j\in M$. Then, it is customary to report tests of $H_{0,j}^M: \beta^M_j = 0$ for each coefficient in the model. If $M$ was chosen in a data-dependent way, then to control selective error we must condition on having selected $(M, H_{0,j}^M)$, which in this case is the same as conditioning on having selected model $M$.

There are many data-driven methods for variable selection in linear regression, ranging from AIC minimization to forward stepwise selection, cf. \citet{hastie2009elements}. We will consider one procedure in particular, based on the lasso, mostly because selective inference in the context of the lasso \citep{lee2016exact} was a main motivation for the present work. The lasso \citep{tibshirani1996regression} provides an estimate of $\beta \in \R^p$ that solves
\begin{equation}
\hat\beta = \underset{\beta}{\text{argmin}}\ ||Y - X\beta||_2^2 + \lambda ||\beta||_1,
\end{equation}
where $X$ is the ``full'' matrix consisting of all $p$ predictors. The first term is the usual least-squares objective, while the second term encourages many of the coefficients to be exactly zero. Because of this property, it makes sense to define the model ``selected'' by the lasso to be the set of variables with non-zero coefficients, i.e.,
\[ \hM(Y) = \{ j: \hat\beta_j \neq 0 \}. \]

Notice that $\hM(Y)$ can take on up to $2^p$ possible values, one for each subset of $\{1, ..., p\}$. The regions $A_M = \{y: \hM(y) = M\}$ form a partition of $\R^n$ into regions that correspond to each model. To control the selective error after selecting a particular $M$, we must condition on the event that $Y$ landed in $A_M$. The partition for a lasso problem with $p=3$ variables in $n=2$ dimensions is shown in Figure \ref{fig:lasso_partition}. An explicit characterization of the lasso partition can be found in \citet{lee2016exact}; see also \citet{harris2014visualizing} for an interactive visualization of the way the lasso partitions the sample space. A different selection procedure would partition the sample space differently; characterizations of the partitions in forward stepwise selection and marginal screening can be found in \citet{loftus2014significance} and \citet{lee2014marginal}, respectively.

\begin{figure}
\begin{center}
\includegraphics[trim = 23mm 23mm 23mm 23mm, clip=TRUE, width=.4\textwidth]{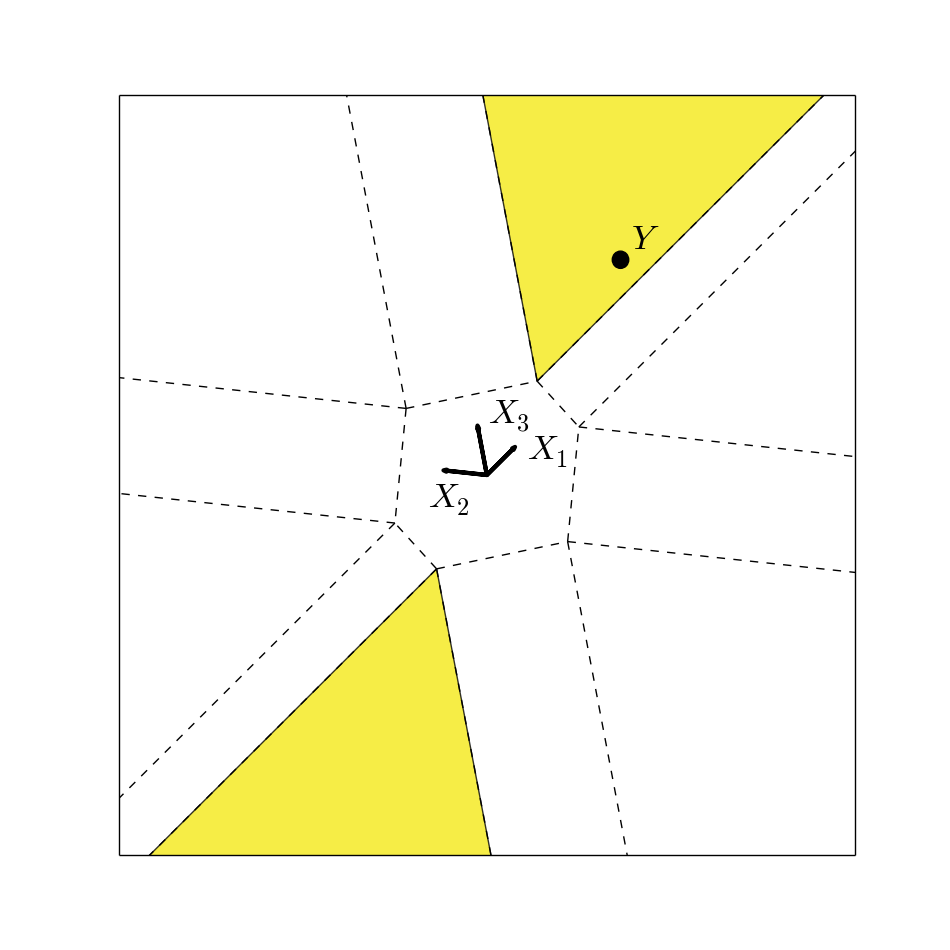}
\end{center}
\caption{An example of the lasso with $n=2$ observations and $p=3$ variables.Tests are based on the distribution of $Y$, conditional on its landing in the highlighted region.}
\label{fig:lasso_partition}
\end{figure}

Imagine that in stage one, we loaded the data into a software package and computed $\hM(Y)$, but we remain otherwise ignorant of the value $Y$ --- that is, we have observed {\em which} of the regions $Y$ falls into but not {\em where} $Y$ is in that region. Now that we have chosen the model, we will construct tests of $H_{0,j}^M: \beta^M_j = 0$ for each of the selected variables. In the example shown in Figure \ref{fig:lasso_partition}, we selected variables $1$ and $3$ and thus test the two hypotheses
\begin{align*}
H^{\{1, 3\}}_{0, 1}&: \beta^{\{1, 3\}}_1 = 0 \\
H^{\{1, 3\}}_{0, 3}&: \beta^{\{1, 3\}}_3 = 0.
\end{align*}
Notice that we have to be careful to always specify the model along with the coefficient, since the coefficient for variable $j$ does not necessarily have a consistent interpretation across different models. Each regression coefficient summarizes the effect of that variable, adjusting for the other variables in the model. For example, ``What is the effect of IQ on salary?'' is a genuinely different question from ``What is the effect of IQ on salary, after adjusting for years of education?''  Both questions are meaningful, but they are fundamentally different.\footnote{We use the word ``effect'' here informally to refer to a regression coefficient, recognizing that regression cannot establish causal claims on its own.}

Having chosen the model $M$ and conditioned on the selection, we will base our tests on the precise location of $Y$, which we do not know yet. Conditionally, $Y$ is not Gaussian, but it does follow an exponential family. As a result, we can appeal to the classical theory of \citet{lehmann1955completeness} to construct tests or confidence intervals for its natural parameters, which are $\beta^M$ if $\sigma^2$ is known, and otherwise are $(\beta^M/\sigma^2,1/\sigma^2)$.

With this concrete example in mind, we will now develop a general framework of selective inference that is much more broadly applicable.
Because we are explicitly allowing models and hypotheses to be random, it is necessary to carefully define our inferential goals.
We first discuss selective inference in the context of hypothesis testing. The closely related developments for confidence intervals will follow in Section~\ref{sec:ci}.

\subsection{Selective Hypothesis Tests}\label{sec:setting}

We now introduce notation that we will use for the remainder of the article. Assume that our data $Y$ lies in some measurable space $(\cY, \sF)$, with unknown sampling distribution $Y\sim F$. The analyst's task is to pose a reasonable probability model $M$ --- i.e., a family of distributions which she believes contains $F$ --- and then carry out inference based on the observation $Y$.

Let $\cQ$ denote the {\em question space} of inference problems $q$ we might tackle. A hypothesis testing problem is a pair $q=(M,H_0)$ of a model $M$ and null hypothesis $H_0$, by which we mean a submodel $H_0\sub M$.\footnote{We identify a ``null hypothesis'' like $H_0:\,\mu(F)=0$ with the corresponding subfamily or ``null model'' ${\{F\in M:\, \mu(F)=0\}}$. This should remind us that the error guarantees of a test do not necessarily extend beyond the model it was designed for.} We write $M(q)$ and $H_0(q)$ for the model and hypothesis corresponding to $q$. Without loss of generality, we assume $H_0(q)$ is tested against the alternative hypothesis $H_1(q)=M(q)\setminus H_0(q)$. To avoid measurability issues, we will assume throughout that $\cQ$ is countable, although our framework can be extended to uncountable $\cQ$ with additional care.

In Section~\ref{sec:lassoRegression} where we test each variable in a selected regression model, the question space is
\[ \cQ = \{ (M, H^M_{0, j}): M \sub \{1, \ldots, p\}, j \in M \}. \]
Note our slight abuse of notation in using $M$ interchangeably to refer both to a subset of variable indices and to the corresponding probability model $\left\{N_n(X_M\beta_M, \sigma^2I_n): \beta\in\R^{|M|}\right\}$.

We model selective inference as a process with two distinct stages:
\begin{description}
\item[1. Selection:] From the collection $\cQ$ of possible questions, the analyst selects a subset $\hcQ(Y)\sub \cQ$ to test, based on the data.
\item[2. Inference:] The analyst performs a hypothesis test of $H_0(q)$ against $M(q)\setminus H_0(q)$ for each $q \in \hcQ(Y)$.
\end{description}

In the case of the simple regression example shown in Figure \ref{fig:lasso_partition}, where we selected variables 1 and 3, $\hcQ$ would consist of the hypotheses for each of the two variables in the model:
\[ \hcQ(Y) = \left\{ \left(\{1, 3\}, H^{\{1, 3\}}_{0, 1}\right), \left(\{1, 3\}, H^{\{1, 3\}}_{0, 3}\right) \right\}. \]

A correctly specified model $M$ is one that contains the true sampling distribution $F$. Importantly, we expressly do not assume that all --- or any --- of the candidate models are correctly specified. Because the analyst must choose $M$ without knowing $F$, she could choose poorly, in which case there may be no formal guarantees on the behavior of the test she performs in stage two. Some degree of misspecification is the rule rather than the exception in most real statistical applications, whether models are specified adaptively or non-adaptively. Our analyst would be in the same position if she were to select a (probably wrong) model using $Y$, then use that model to perform a test on new data $Y^*$ collected in a confirmatory experiment. See Section~\ref{sec:modelWrong} for further discussion of this issue.

For our purposes, a {\em hypothesis test} is a function $\phi(y)$ taking values in $[0, 1]$, representing the probability of rejecting $H_0$ if $Y=y$. In most cases, the value of the function will be either 0 or 1, but with discrete variables, randomization may be necessary to achieve exact level $\alpha$.

To adjust for selection in testing $q$, we condition on the event that the question was asked, which we describe by the selection event
\begin{equation}
A_q = \{q \in \hcQ(Y) \},
\label{eq:A_q}
\end{equation}
i.e., the event that $q$ is among the questions asked. In general, the selection events for different questions are not disjoint. In the regression example, where we test $H^M_{0, j}$ if and only if model $M$ is selected, conditioning on $A_q$ is equivalent to simply conditioning on $\hM$. By convention we take $\phi_q(Y) = 0$ for $Y\notin A_q$ to reflect the idea that if a hypothesis is not tested then it is not rejected; note this convention does not affect the selective properties of $\phi_q$.

In selective inference, we are mainly interested in the properties of a test $\phi_q$ for a question $q$, conditional on $A_q$. We say that $\phi_q$ controls {\em selective type I error} at level $\alpha$ if
\begin{equation}\label{eq:selT1}
\E_F\left[\phi_q(Y) \gv A_q\right] \leq \alpha, \;\; \text{ for all } F\in H_0(q).
\end{equation}
and define its {\em selective power function} as
\begin{equation}\label{eq:selPow}
  \pow_{\phi_q}(F \gv A_q) = \E_{F}[\phi_q(Y) \gv A_q].
\end{equation}
Because $\cQ$ is countable, the only relevant $q$ are those for which $\P(A_q) > 0$.

Notice that only the model $M(q)$ and hypothesis $H_0(q)$ are relevant for defining the selective level and power of a test $\phi_q$. This means that in designing valid $\phi_q$, we can concentrate on one $q$ at a time, even if there are many mutually incompatible candidate models in $\cQ$. As long as each $\phi_q$ controls the selective error at level $\alpha$ given its selection event $A_q$, then a global error is also controlled:
\begin{equation}\label{eq:oneAtTime}
  \frac{\E\left[\text{\# false rejections}\right]}{\E\left[\text{\# true nulls selected} \right]} \leq \alpha,
\end{equation}
provided that the denominator is finite. Equation~(\ref{eq:oneAtTime}) holds for countable $\cQ$ regardless of the dependence structure across different $q$. The fact that we can design tests one $q$ at a time makes it much easier to devise selective tests in concrete examples, which we take up in Sections~\ref{sec:exfam}--\ref{sec:examples}.

\subsection{Comparison to Familywise Error Rate}

Selective error control is neither weaker nor stronger than control of the familywise error rate (FWER), which is the probability of rejecting any true null hypothesis:
\begin{equation}
  \text{FWER} = \P_F(\phi_q(Y) = 1 \text{ for any } q \text{ with } H_0(q) \ni F).
\end{equation}
Although the FWER is usually considered the most conservative control guarantee, it does not scale easily across different researchers: suppose that in Example \ref{ex:file_drawer}, each observation $Y_i$ is collected by a different scientific research team at a different university, with each team then publishing the nominal level-$\alpha$ test if $|Y_i|>1$ and otherwise moving on to another project. At the level of a single research group and experiment, FWER is controlled, but there is a major unaccounted-for multiplicity problem if we consider the discipline as a whole.

By contrast, selective error control scales naturally across multiple research groups and requires no coordination among groups. If each research team in a discipline controls the selective error rate for each of its own experiments, then the discipline as a whole will achieve long-run control of the type I error rate among true {\em selected} null hypotheses, just as they would if there were no selection.
\begin{proposition}[Discipline-Wide Error Control]\label{prop:disciplineWide}
Suppose there are $n$ independently operating research groups in a scientific discipline with a shared, countable question space $\cQ$. Research group $i$ collects data $Y_{i}\sim F_{i}$, applies selection rule $\hcQ_{i}(Y_{i})\sub \cQ$, and carries out selective level-$\alpha$ tests $(\phi_{q,i}(y_{i}), q\in\hcQ_{i})$.
Assume each research group has probability at least $\delta>0$ of carrying out at least one test of a true null, and for some common $B<\infty$,
\[\E_{F_{i}}\left[|\hcQ_{i}(Y_{i})|^2\right] \leq B,
\quad\text{ for all } i.\]
Then as $n$ grows, the discipline as a whole achieves long-run control over the frequentist error rate
\begin{equation}\label{eq:longRun}
  \limsup_{n\to\infty} \frac{\textnormal{\# false rejections}}{\textnormal{\# true nulls selected}} \leqAS \alpha.
\end{equation}
\end{proposition}
The proof is deferred to Appendix~\ref{sec:disciplineWideProof}. 

There is no counterpart to Proposition~\ref{prop:disciplineWide} for other popular error rates such as the false discovery rate (FDR) \citep{benjamini1995controlling} or familywise error rate (FWER). Section~\ref{sec:multiple} discusses further the relationship between selective error control and other common error rates in multiple inference.

\subsection{Selective Confidence Intervals}\label{sec:ci}

If the goal is instead to form confidence intervals for a parameter $\theta(F)$, it is more convenient to think of $\cQ$ as containing pairs $q = (M, \theta(\cdot))$ of a model and a parameter. By analogy to~(\ref{eq:selT1}), we will call a set $C(Y)$ a {\em $(1-\alpha)$ selective confidence set} if
\begin{equation}\label{eq:selConf}
  \P_F(\theta(F) \in C(Y) \gv A_q) \geq 1 - \alpha, \;\;\text{ for all } F\in M.
\end{equation}

The next result establishes that selective confidence sets can be obtained by inverting selective tests, as one would expect by analogy to the classical case.

\begin{proposition}[Duality of Selective Tests and Confidence Sets]
Suppose we form a confidence interval for $\theta(F)$ on the event $A_q$. Suppose also that on this event, we form a test $\phi_t$ of $H_{0, t} = \{F:\,\theta(F) = t\}$ for all $t$. Let $C(Y)$ be the set of $t$ for which $\phi_t$ does not (always) reject:
\begin{equation}
  C(Y) = \left\{t: \phi_t(Y) < 1 \right\}.
\end{equation}
If each $\phi_t$ is a selective level-$\alpha$ test, then $C(Y)$ is a selective $(1-\alpha)$ confidence set.
\end{proposition}
\begin{proof}
The selective non-coverage probability is
\begin{equation*}
\P_F(\theta(F) \notin C(Y) \gv A_q) \ = \ \P_F(\phi_{\theta(F)}(Y) = 1 \gv A_q) \ \leq\ \E_F\left[\phi_{\theta(F)}(Y) \gv A_q\right] \ \leq\ \alpha.
\end{equation*}
\end{proof}

\subsection{Conditioning Discards Information}\label{sec:consumes}

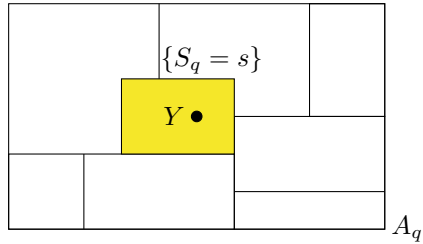
\begin{figure}
\begin{center}
\begin{tikzpicture}
\draw (0,0) -- (5,0) -- (5,3) -- (0,3) -- cycle;
\draw (0,0) -- (3,0) -- (3,1) -- (0,1) -- cycle;
\draw (3,0) -- (3,1.5) -- (5,1.5) -- (5,0) -- cycle;
\draw (4,1.5) -- (4,3) -- (5,3) -- (5,1.5) -- cycle;
\draw (2,2) -- (2,3);
\draw (1,0) -- (1,1);
\draw (3,0.5) -- (5,0.5);
\draw[fill=yellow!90!gray] (1.5,1) -- (1.5,2) -- (3,2) -- (3,1) -- cycle;
\draw[fill] (2.5,1.5) circle (2pt);
\node at (2.2,1.5) {$Y$};
\node at (5.3,0) {$A_q$};
\node at (2.7, 2.25) {$\{ S_q = s \}$};
\end{tikzpicture}
\end{center}
\caption{Instead of conditioning on the selection event $A_q$ that question $q$ is asked, we can condition on a finer event, the value of the random variable $S_q$. We call $S_q$ the {\em selection variable}.}
\label{fig:finer}
\end{figure}

Because performing inference conditional on a random variable effectively disqualifies that variable as evidence against a hypothesis, we will typically want to condition on as little data as possible in stage two. Even so, some selective inference procedures condition on more than $A_q$. For example, \sampOrData splitting can be viewed as inference conditional on $Y_{1}$, the part of the data used for selection. More generally, we say a {\em selection variable} is any variable $S_q(Y)$ whose level sets partition the sample space more finely than $A_q$ does; i.e., $A_q \in \sF(S_q)$. Informally, we can think of conditioning on a finer partition of $A_q$, as shown in Figure~\ref{fig:finer}.

We say $\phi$ controls the {\em selective type I error with respect to $S_q$} at level $\alpha$ if the error rate is less than $\alpha$ given $S_q=s$ for $\{S_q=s\}\sub A_q$. More formally,
\begin{equation}\label{eq:selT1S}
  \E_F\left[\phi(Y)\1_{A_q}(Y) \gv S_q \right] \leqAS \alpha, \;\;\text{ for all } F\in H_0(q)
\end{equation}
Taking $S_q(y) = \1_{A_q}(y)$, the coarsest possible selection variable, recovers the baseline selective type I error in~(\ref{eq:selT1}). The definition of a selective confidence set may be generalized in the same way.

Generalizing~(\ref{eq:sampleCarving}) to finer selection variables gives
\begin{equation}\label{eq:sampleCarvingS}
  \sF_0 \underlabel_{\text{used for selection}} \sF(S(Y))
  \underlabel_{\text{used for inference}} \sF(Y),
\end{equation}
suggesting that the more we refine $S(Y)$, the less data we have left for second-stage inference. Indeed, the finer $S$ is, the more stringent is the requirement~(\ref{eq:selT1S}):
\begin{proposition}[Monotonicity of Selective Error]\label{prop:fineCont}
  Suppose $\sF(S_{1}) \sub \sF(S_{2})$. If $\phi$ controls the type I error rate at level $\alpha$ for $q=(M,H_0)$ w.r.t. the finer selection variable $S_{2}$, then it also controls the type I error rate at level $\alpha$ w.r.t. the coarser $S_{1}$.
\end{proposition}
\begin{proof}
If $F\in H_0$, then
\begin{equation*}
  \E_F\left[\phi(Y)\1_{A}(Y) \gv S_{1}\right] \;\;=\;\; \E_F\left[\;\E_F\big[\phi(Y)\1_{A}(Y) \gv S_{2}\big] \gv S_{1}\right] \;\;\leqAS\;\; \alpha.
\end{equation*}
\end{proof}
Because $S(y) = \1_A(y)$ is the coarsest possible choice, a test controlling the type I error w.r.t. any other selection variable also controls the selective error in~(\ref{eq:selT1}). At the other extreme, if $S(y) = y$, then we cannot improve on the trivial ``coin-flip'' test $\phi(y)\equiv \alpha$. Proposition~\ref{prop:fineCont} suggests that we will typically sacrifice power as we move from coarser to finer selection variables.
Even so, refining the selection variable can be useful for computational reasons. For example, in the case of the lasso, by conditioning additionally on the signs of the nonzero $\hat\beta_j$, the selection event becomes a convex region instead of the union of up to $2^{|\hM|}$ disjoint convex regions \citep{lee2016exact}. Another valid reason to refine $S_q$ beyond $\1_{A_q}$ is to strengthen our inferential guarantees in a meaningful way; for example, we can achieve achieve false coverage-statement rate (FCR) control by choosing $S_q=(\1_{A_q}(Y),|\hcQ(Y)|)$ (see Section~\ref{sec:multiple}, Proposition~\ref{prop:fcrCont}).

\capSampOrData splitting corresponds to setting every selection variable equal to $S = Y_{1}$. As a result, \sampOrData splitting does not use all the information that remains after conditioning on $A$, as we see informally in the three-stage filtration
\begin{equation}\label{eq:ssWaste}
  \sF_0 \underlabel_{\text{used for selection}} \sF(\1_{A}(Y_{1}))
  \underlabel_{\text{wasted}} \sF(Y_{1})
  \underlabel_{\text{used for inference}} \sF(Y_{1}, Y_{2}).
\end{equation}
As we will see in Section~\ref{sec:admissibility}, this waste of information means that \sampOrData splitting is inadmissible under fairly general conditions.

We can quantify the amount of leftover information in terms of the Fisher information that remains in the conditional law of $Y$ given $S$. In a smooth parametric model, we can decompose the Hessian of the log-likelihood as
\begin{equation}\label{eq:infoPartition}
  \grad^2 \ell(\theta;\;Y) = \grad^2 \ell(\theta; \;S) + \grad^2 \ell(\theta; \;Y \gv S)
\end{equation}
The conditional expectation
\begin{equation}
  \cI_{Y\gv S}(\theta; S) = -\E\left[\grad^2 \ell(\theta;\;Y \gv S) \gv S\right]
\end{equation}
is the {\em leftover Fisher information} after selection at $S(Y)$ (the leftover information is essentially the same as the missing information of \citet{orchard1972missing}, but we find ``leftover'' to be a more intuitive descriptor than missing in this context since the information is at our disposal). 

Taking expectations in~(\ref{eq:infoPartition}), we obtain
\begin{equation}
  \E\left[\cI_{Y\gv S}(\theta; S)\right] \;=\; \cI_Y(\theta) - \cI_S(\theta) \;\preceq\; \cI_Y(\theta).
\end{equation}
Thus, on average, the price of conditioning on $S$ --- the price of selection --- is the information $S$ carries about $\theta$.\footnote{Note that we do not necessarily have $\cI_{Y\gv S}(\theta; S) \preceq \cI_Y(\theta)$ for every $S$. In fact there are interesting counterexamples where $\cI_{Y\gv A}(\theta) \gg \cI_{Y}(\theta)$ for certain $\theta$, but we will not take them up here.}
 In some cases this loss may be quite small, which a simple example elucidates.

\begin{example}\label{ex:univar}
Consider selective inference under the univariate Gaussian model
\begin{equation}
  Y \sim N(\mu, 1),
\end{equation}
after conditioning on the selection event $A = \{Y > 3\}$.

Figure~\ref{fig:infoUnivar} plots the leftover information as a function of $\mu$. If $\mu\ll 3$, there is very little information in the conditional distribution: whether $\mu = -10$ or $\mu=-11$, $Y$ is conditionally highly concentrated on 3. By contrast, if $\mu \gg 3$, then $\P_\mu(A) \approx 1$, the conditional law is practically no different from the marginal law, and virtually no information is lost in the conditioning.

Figure~\ref{fig:umpuUnivar} shows the confidence intervals that result from inverting the tests described in Section~\ref{sec:exfam}. When $Y\gg 3$, the interval essentially coincides with the nominal interval $Y \pm 1.96$ because there is hardly any selection bias and no real adjustment is necessary. By contrast, when $Y$ is close to 3 it is potentially subject to severe selection bias. This fact is reflected by the confidence interval, which is both longer than the nominal interval and centered at a value significantly less than $Y$.
\end{example}

\captionsetup[subfigure]{width=1.2\textwidth}
\begin{figure}
  \centering
  \begin{subfigure}[t]{0.4\textwidth}
    \includegraphics[width=\textwidth]{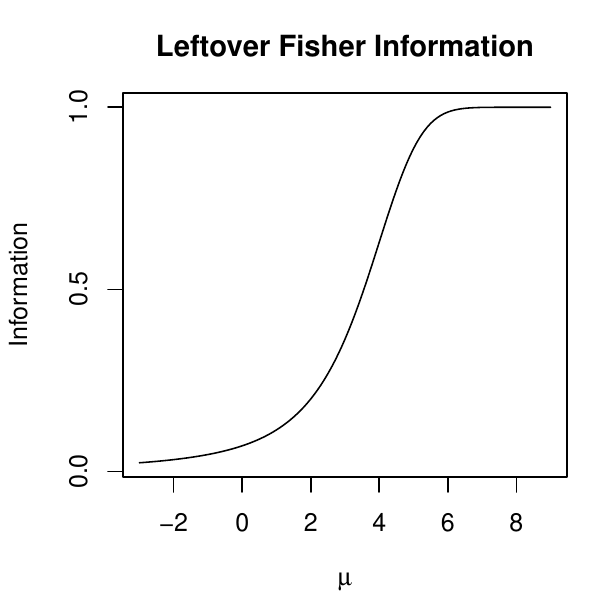}
    \caption{Leftover Fisher information as a function of $\mu$. For $\mu\ll 3$, then there is very little information in the conditional distribution, since $Y$ is conditionally highly concentrated on 3. For $\mu \gg 3$, then $\P_\mu(A) \approx 1$ and virtually no information is lost.}
    \label{fig:infoUnivar}
  \end{subfigure}
  \hspace{.1\textwidth}
  \begin{subfigure}[t]{0.4\textwidth}
    \includegraphics[width=\textwidth]{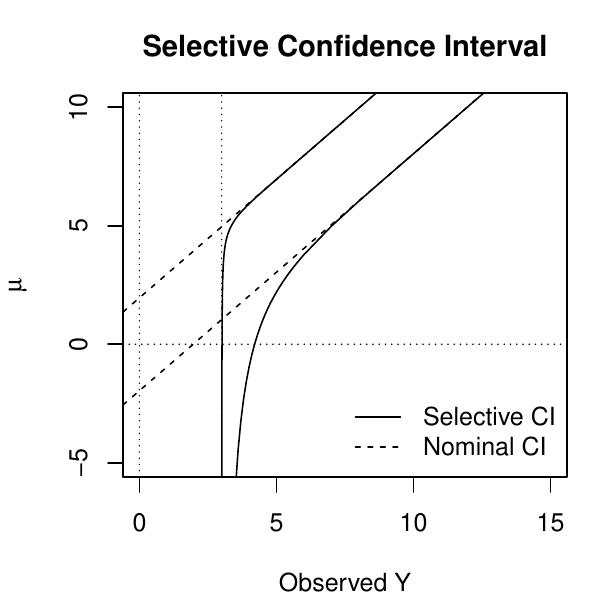}
    \caption{Confidence intervals from inverting the UMPU tests of Section~\ref{sec:exfam}. For $Y\gg 3$, the interval essentially coincides with the nominal interval ${Y \pm 1.96}$. For $Y$ close to 3, the wide interval reflects potentially severe selection bias.}
    \label{fig:umpuUnivar}
  \end{subfigure}
  \caption{Univariate Gaussian. $Y\sim N(\mu, 1)$ with selection event $A=\{Y>3\}$.}
\label{fig:univar}
\end{figure}

\subsection{Conceptual Questions}\label{sec:conceptual}

We now pause to address conceptual objections we have encountered when explaining our work. These objections can be most easily expressed, and answered, in the setting where there is a single selected model and a single selected hypothesis to test or confidence interval to construct (i.e., $\hcQ$ is always a singleton).

There is a common theme in every one of the conceptual objections to follow: they are all equally good grounds for objecting to data splitting, or for that matter, to selecting a model and hypothesis based on a prior experiment whose outcome was random. Thus, a good exercise is to ask ourselves how we would answer the same question if it were asked about data splitting; most likely, the same answer applies equally well to data carving.

\subsubsection{How can the model be random?}\label{sec:howRandom}

In our framework, inference is based on a statistical model $M$ that is allowed to be chosen randomly, based on the data $Y$. A common first reaction is that if the data are generated according to the model, and the model is selected based on the data, then the whole business is circular and nonsensical.

To resolve this conundrum, note that in our framework the {\em true} sampling distribution $F$ is not selected in any sense; it is entirely outside the analyst's control. The only thing selected is the {\em working model}, a tool the analyst uses to carry out inference, which may or may not include the true $F$. Thus, the sampling distribution $F$ comes first, then the data $Y$, then the model $M$. 

Random selection of models is not new and should not trouble or confuse us: $M$ would be just as random if it were selected via data splitting, or for that matter if it were based on a prior experiment.
Thoughtful skeptics may find reasons for concern about all of these approaches, believing that statistical testing is only appropriate when a model can be based purely on convincing theoretical considerations. We answer only that this point of view would rule out most scientific inquiries for which statistics is ever used. However, for those who are comfortable with choosing a random model using \sampOrData splitting or a previous experiment, we see no special reason to be any more concerned about choosing a random model using \sampOrData carving.

In any case, $M$ is by no means required to be random, and our conditional-inference framework applies in many interesting settings where $M$ is always the same pre-specified parametric or nonparametric model, but we adaptively choose which hypotheses to test or which parameters to estimate. For example, in our clinical trial example of Section~\ref{sec:clinical}, the statistical model is always the same but we choose which null hypotheses to test after inspecting the data. The same is true of the conditional confidence intervals of \citet{weinstein2013selection}, the saturated-model selective $z$-test proposed by \citet{lee2016exact} and discussed in~\ref{sec:fullModel}, and the rank verification methods proposed in \citet{hung2016rank}.

\subsubsection{What if the selected model is wrong?}\label{sec:modelWrong}

If we were writing about a topic other than selective inference, we might have begun by stating a formal mathematical assumption that the sampling distribution $F$ belongs to a known model $M$, and then devised a test $\phi$ that behaves well when $F\in M$. The same $\phi$ might not work well at all for $F\notin M$: for example, if we choose to apply the one-sample $t$-test of $\mu=0$ to a sample $Y_1,\ldots,Y_n$ whose observations are highly correlated, then the probability of rejection may be a great deal larger than the nominal $\alpha$, even if $\E[Y_i]=0$. This is not a mistake in the formal theory, nor does it make the $t$-test an inherently invalid test; rather, the validity or invalidity of a test is defined with respect to its behavior when $F\in H_0 \sub M$.

In any given application, the analyst must choose from among many statistical methods knowing that each one is designed to work under a particular set of parametric or nonparametric assumptions about $F$ --- i.e., under a particular model $M$. Because our theory encompasses both the choice and the subsequent analysis, it would not be sensible to assume that the analyst is infallible and always selects a correct model. Typically some candidate models $M$ are correctly specified, others are not, and the analyst can never know for sure which are which. Any model selection procedure using data splitting, data carving, or a prior experiment always carries a risk of selecting the wrong model, and in all cases the second-stage type I error guarantees are only in force when the model is correct.

Of course, the possibility of misspecification is not restricted to adaptive procedures like \sampOrData carving and \sampOrData splitting: selecting an inappropriate model {\em after} seeing the data leaves us no better or worse off than if we had chosen the same inappropriate model {\em before} seeing the data. The alternative to {\em adaptive} model selection is not {\em infallible} model selection, it is {\em non-adaptive} model selection.

There is a separate question of robustness: if $F\notin M$ but is ``close'' in some sense, we may still want our procedure to behave predictably. However, even if some model gives a reasonable approximation to $\cL(Y)$, there is no guarantee that the induced model for $\cL(Y \gv A)$ is reasonable, since conditioning can introduce new robustness problems.
For example, suppose that a test statistic $Z_n(Y)$ tends in distribution to $N(0,1)$ under $H_0$ as $n\to \infty$. In a non-selective setting, we might be comfortable modeling it as Gaussian as a basis for hypothesis testing. In this case it is also true that ${\cL(Z_n \gv Z_n > c)}$ converges to a truncated Gaussian law for any fixed $c\in \R$, but the approximation may be much poorer for intermediate values of $n$. Worse, if we use increasing thresholds $c_n\to\infty$ with $n$, the truncated Gaussian approximation may never become reasonable. Understanding the interaction between selective inference and asymptotic approximations is an area of active ongoing study; see \citet{tian2017asymptotics,tibshirani2015uniform,tian2015selective,taylor2016post} for subsequent works discussing asymptotics without Gaussian assumptions.

\subsubsection{Does the result have a marginal interpretation?}\label{sec:margInterp}

Consider an adaptive clinical trial in which we select the most promising subgroup of patients based on some preliminary analysis, and then report a confidence interval for the average treatment effect on that subgroup. For some realizations of the data, we might decide to return an interval for the effect on men over the age of 45, and for other realizations we might decide to return an interval for the effect on Hispanic women with high blood pressure. Let $S(Y)$ denote the selected subpopulation, a random region of covariate space, let $\theta(s)$ denote the true average treatment effect on a given subpopulation $s$, and let $C_s(Y)$ denote the confidence interval we construct for $\theta(s)$ when $S(Y)=s$.

We find that confusion often occurs when people attempt to interpret $C(Y) = C_{S(Y)}(Y)$ as a marginal confidence interval for $\theta(S(Y))$, the treatment effect on a {\em random} subpopulation. If we ran the experiment again with the same selection procedure, we might choose a completely different $S(Y)$, giving $C(Y)$ a completely different meaning, and 100 realizations of the data might produce 100 disjoint realizations of $C(Y)$, meant to cover 100 very different true parameter values.

While technically correct, the above interpretation is usually best avoided. Rather, we recommend thinking of each $C_s$ as a different confidence interval for a different {\em fixed} parameter $\theta(s)$, having nothing to do with $C_{s'}$ for $s'\neq s$. During the selection stage, we choose one $C_s$ to construct and leave the other intervals undefined.

In other words, the interval has no useful interpretation until the first stage is complete, after which $S(Y)$ is fixed: it is pointless to try to interpret an answer before we even decide what question to ask. It is true that we might have asked about a different parameter if the data had looked different. By the same token, we might have performed an entirely different experiment if our most recent grant application had been funded. Neither of these contingencies should be a source of confusion because experiments not performed, or parameters not selected, are irrelevant to the situation at hand.

\subsection{Prior Work on Selective Inference}\label{sec:prior}

This article takes its main inspiration from a
recent ferment of work on the problem of inference in linear regression models after model selection. \citet{lockhart2014significance} derive an asymptotic test for whether the nonzero fitted coefficients at a given knot in the lasso path contain all of the true nonzero coefficients. \citet{tibshirani2014exact} provided an exact (finite-sample) version of this result and extended it to the LARS path, while \citet{lee2016exact}, \citet{loftus2014significance}, and \citet{lee2014marginal} used similar approaches to derive exact tests for the lasso with a fixed value of regularization parameter $\lambda$, forward stepwise regression, and regression after marginal screening, respectively. All of the above approaches are derived assuming that the error variance $\sigma^2$ is known or an independent estimate is available.

The present work attempts to unify the above approaches under a common theoretical framework generalizing the classical optimality theory of \citet{lehmann1955completeness}, and elucidate previously unexplored questions of power. It also lets us generalize the results to the case of unknown $\sigma^2$, and to arbitrary exponential families after arbitrary selection events. Since the initial appearance of this work, it has been applied in many other settings; see \citet{taylor2015statistical} for a recent review.

Other works have viewed selective inference as a multiple inference problem. Recent work in this vein can be found in \citet{berk2013valid} and \citet{barber2015controlling}. Section~\ref{sec:multiple} argues that inference after model selection and multiple inference are distinct problems with different scientific goals; see \citet{benjamini2010simultaneous} for more discussion of this distinction. An empirical Bayes approach for selection-adjusted estimation can be found in \citet{efron2011tweedie}.

There has also recently been work on inference in high-dimensional linear regression models, notably \citet{belloni2011inference}, \citet{belloni2014inference}, \citet{zhang2014confidence}, \citet{javanmard2014hypothesis}, and \citet{van2014asymptotically}; see \citet{dezeure2015high} for a review. These works focus on approximate asymptotic inference for a fixed model with many variables, while we consider finite-sample inference after selecting a smaller submodel to focus our inferential goals.

\citet{leeb2005model,leeb2006can,leeb2008can} prove certain impossibility results regarding estimating the distribution of post-selection estimators. These results do not apply to our framework; under the statistical models we use, the post-selection distributions of our test statistics are known and thus do not require estimation.

The foregoing works are frequentist, as is this work. Because Bayesian inference conditions on the entire data set, conditioning first on a selection event typically has no operative effect on the posterior: if $p$ and $\pi$ are respectively the marginal likelihood and prior, then $p(Y \gv A, \theta)\cdot \pi(\theta \gv A) \propto p(Y \gv \theta)\cdot \pi(\theta)$ for $Y\in A$  \citep{dawid1994selection}. \citet{yekutieli2012adjusted} argues that in certain cases it is more appropriate to condition the likelihood on selection without changing the prior to reflect that conditioning, resulting in a posterior proportional to $p(Y \gv A, \theta)\cdot \pi(\theta)$. The credible intervals discussed in \citet{yekutieli2012adjusted} resemble the confidence intervals proposed in this article, and the discussion therein presents a somewhat different perspective on how and why conditioning can adjust for selection.

Though our goals are very different, our theoretical framework is in some respects similar to the conditional confidence framework of \citet{kiefer1976admissibility}, in which inference is made conditional on some estimate of the confidence with which a decision can be made. See also \citet{kiefer1977conditional, brownie1977ideas,brown1978contribution,berger1994unified}.

\citet{olshen1973conditional} discussed error control given selection in a two-stage multiple comparison procedure, in which an $F$-test is first performed, then Scheff\'{e}'s $S$-method applied if the $F$-test rejects. For large enough rejection thresholds, simultaneous coverage in the second stage is less than $1-\alpha$ conditional on rejection in stage one.

\section{Selective Inference in Exponential Families}\label{sec:exfam}

As discussed in Section~\ref{sec:setting}, we can construct selective tests ``one at a time'' for each model--hypothesis pair $(M, H_0)$, conditional on the corresponding selection event $A_q$ and ignoring any other models that were previously under consideration. This is because the other candidate models and hypotheses are irrelevant to satisfying~(\ref{eq:selT1}). For that reason, we suppress the explicit dependence on $q=(M,H_0)$ except where it is necessary to resolve ambiguity.

Our framework for selective inference is especially convenient when $M$ corresponds to a multiparameter exponential family
\begin{equation}\label{eq:exfam}
  Y \sim f_{\theta}(y) = \exp\{\theta'T(y) - \psi(\theta)\}\, f_0(y)
\end{equation}
with respect to some dominating measure. Then, the conditional distribution given $Y\in A$ for any measurable $A$ is another exponential family with the same natural parameters and sufficient statistics but different carrier measure and normalizing constant:
\begin{equation}
  (Y\gv Y\in A) \sim \exp\{\theta'T(y) -  \psi_A(\theta)\}\,
  f_0(y)\, \1_A(y)
\end{equation}
This fact lets us draw upon the rich theory of inference in
multiparameter exponential families.

\subsection{Conditional Inference and Nuisance Parameters}\label{sec:condInf}
Classically, conditional inference in exponential families arises as a means for inference in the presence of nuisance parameters, as in Model~\ref{mod:partExFam} below.

\begin{model}[Exponential Family with Nuisance Parameters]\label{mod:partExFam}
$Y$ follows a $p$-parameter exponential family with sufficient statistics $T(y)$ and $U(y)$, of dimension $k$ and $p-k$ respectively:
\begin{equation}\label{eq:partExFam}
  Y\sim f_{\theta,\zeta}(y)
  = \exp\{\theta'T(y) + \zeta 'U(y) - \psi(\theta, \zeta)\}\, f_{0}(y),
\end{equation}
with $(\theta,\zeta)\in \Theta\sub \R^{p}$ open.
\end{model}

Assume $\theta$ corresponds to a parameter of interest and $\zeta$ to an unknown nuisance parameter.
The conditional law $\cL(T(Y) \gv U(Y))$ depends only on $\theta$:
\begin{equation}\label{eq:efCond}
  \left(T\gv U=u\right) \sim g_{\theta}(t \gv u) = \exp\{\theta't -
    \psi_g(\theta\gv u)\}\, g_0(t\gv u),
\end{equation}
letting us eliminate $\zeta$ from the problem by conditioning on $U$. For $k=1$ (i.e., for $\theta\in \R$), we obtain a single-parameter family for $T$.

Consider testing the null hypothesis $H_0:\; \theta \in \Theta_0 \sub \Theta$ against the alternative $H_1:\; \theta \in \Theta_1 = \Theta\setminus \Theta_0$. We say a level-$\alpha$ selective test $\phi(y)$ is {\em selectively unbiased} if
\begin{equation}\label{eq:selUnbi}
  \pow_\phi(\theta \gv A) = \E_{\theta}[\phi(Y) \gv A] \geq \alpha, \quad \text{ for all }\theta\in \Theta_1.
\end{equation}
The condition (\ref{eq:selUnbi}) specializes to the usual definition of an unbiased test when there is no selection (when $A=\cY$). Unbiasedness rules out tests that privilege some alternatives to the detriment of others, such as one-sided tests when the alternative is two-sided.

A {\em uniformly most powerful unbiased} (UMPU) selective level-$\alpha$ test is one whose selective power is uniformly highest among all level-$\alpha$ tests satisfying~(\ref{eq:selUnbi}). A selectively unbiased confidence region is one that inverts a selectively unbiased test, and confidence regions inverting UMPU selective tests are called uniformly most accurate unbiased (UMAU). All of the above specialize to the usual definitions when $A=\cY$.

See \citet{lehmann2005testing} or \citet{brown1986fundamentals} for thorough reviews of the rich literature on testing in exponential family models. In particular, the following classic result of \citet{lehmann1955completeness} gives a simple construction of UMPU tests in exponential family models.
\begin{theorem}[\citet{lehmann1955completeness}]\label{thm:margUmpu}
  Under Model~\ref{mod:partExFam}
  with $k=1$, consider testing the hypothesis
  \begin{equation}\label{eq:uniNull}
    H_0:\,\theta = \theta_0 \quad \text{ against }
    \quad H_1:\,\theta \neq \theta_0
  \end{equation}
  at level $\alpha$.
  There is a UMPU test of the form $\phi(Y) = f(T(Y), U(Y))$ with
  \begin{equation}\label{eq:testForm}
    f(t,u) = \begin{cases}
        1 & t < c_1(u) \text{ or } t > c_2(u)\\
        \gamma_i & t = c_i(u)\\
        0 & c_1(u) < t < c_2(u)
      \end{cases}
  \end{equation}
  where $c_i$ and $\gamma_i$ are chosen to satisfy
   \begin{align}\label{eq:margSolveLev}
     \E_{\theta_0}\left[f(T, U)\gv U=u\right] &= \alpha \\
     \label{eq:margSolveUnbi}
     \E_{\theta_0}\left[T f(T, U)\gv U=u\right]
     &=  \alpha\,\E_{\theta_0}\left[T\gv U=u\right].
   \end{align}
\end{theorem}
The condition (\ref{eq:margSolveLev}) constrains the power to be $\alpha$ at
$\theta=\theta_0$, and (\ref{eq:margSolveUnbi}) is obtained by
differentiating the power function and setting its derivative
to 0 at $\theta=\theta_0$.

Because $\cL( Y \gv A)$ is an exponential family, we can simply apply Theorem~\ref{thm:margUmpu} to the conditional law $\cL( Y \gv A)$ to obtain an analogous construction in the selective setting.

\begin{corollary}[UMPU Selective Tests]\label{cor:selUmpu}
   Under Model~\ref{mod:partExFam} with $k=1$, consider testing the hypothesis
   \begin{equation}\label{eq:uniNullSel}
     H_0:\,\theta = \theta_0 \quad \text{ against }
     \quad H_1:\,\theta     \neq \theta_0
   \end{equation}
   at selective level $\alpha$ on selection event $A$.
   There is a UMPU selective test of the form $\phi(Y) = f(T(Y),U(Y))$ with
   \begin{equation}
     f(t,u) = \left\{\begin{matrix}
         1 & t < c_1(u) \text{ or } t > c_2(u)\\
         \gamma_i & t = c_i(u)\\
         0 & c_1(u) < t < c_2(u)
       \end{matrix}\right.
   \end{equation}
   for which $c_i$ and $\gamma_i$ solve
   \begin{align}\label{eq:selSolveLev}
     \E_{\theta_0}\left[f(T, U)\gv  U=u,\; Y\in A\right] &= \alpha \\
     \label{eq:selSolveUnbi}
     \E_{\theta_0}\left[Tf(T, U)\gv U=u,\; Y\in A \right]
     &=  \alpha\,\E_{\theta_0}\left[T\gv U=u,\; Y\in A\right].
   \end{align}
\end{corollary}

We emphasize here that the test $\phi$ as defined above is not merely UMPU among selective tests that condition on $U$, but rather it is UMPU among {\em all} selective level-$\alpha$ tests; see \citet{lehmann2005testing} for more details. In some cases, it may be useful to interpret $\phi$ conditionally on $U=u$, for example if the observed $u$ leads to a more or less powerful test.

It is worth keeping in mind that unbiasedness is only one way to choose a test when there is no completely UMP one. For example, another simple choice is to use the equal-tailed test from the same conditional law \eqref{eq:efCond}. The equal-tailed level-$\alpha$ rejection region is simply the union of the one-sided level-$\alpha/2$ rejection regions. While the equal-tailed and UMPU tests choose $c_i$ and $\gamma_i$ in different ways, both tests take the form~(\ref{eq:testForm}). In fact, as we will see next, {\em all} admissible tests are of this form, which implies that data splitting tests are usually inadmissible.

\subsection{Conditioning, Admissibility, and Data Splitting}\label{sec:admissibility}

A selective level-$\alpha$ test $\phi$ is {\em inadmissible} on selection event $A$ if there exists another selective level-$\alpha$ test $\phi^*$ for which
\begin{equation}
  \E_{\theta,\zeta}[\phi^*(Y) | A] \geq \E_{\theta,\zeta}[\phi(Y) | A], \;\;\text{ for all } (\theta,\zeta)\in \Theta_1,
\end{equation}
with the inequality strict for at least one $(\theta,\zeta)$. In the main result of this section, we will show that tests based on data splitting are nearly always inadmissible.

Let $Y$ be an observation from Model~\ref{mod:partExFam}, and suppose we wish to test
 \begin{equation}\label{eq:hypEq}
  H_0:\;\theta=\theta_0 \quad \text{ against }
  \quad H_1:\; \theta\neq \theta_0.
\end{equation}
We will assume all tests are functions of the sufficient statistic and write (with some abuse of notation) $\phi(T, U)$ for $\phi(Y)$. We can do this without loss of generality because any test $\phi(Y)$ can be Rao-Blackwellized, i.e.,
\[ \phi(T, U) \equiv \E[\phi(Y) | T, U], \]
to obtain a new test that is a function of $(T, U)$, with the same power function as the original. Therefore, if $\phi(T, U)$ is inadmissible, then so is the original test $\phi(Y)$.

Now we can apply the following result of \citet{matthes1967tests}.

\begin{theorem}[Matthes and Truax, Theorem 3.1]
Let $Y$ be an observation from Model~\ref{mod:partExFam}, and suppose we wish to test
 \begin{equation}\label{eq:hypEq}
  H_0:\;\theta=\theta_0 \quad \text{ against }
  \quad H_1:\; \theta\neq \theta_0.
\end{equation}
 Let $\sC$ denote the class of all level-$\alpha$ tests $\phi(T, U)$ of the form
 \begin{equation}\label{eq:convexU}
   \phi(t, u) = \begin{cases}
       0 & t \in \text{int } C(u) \\
       \gamma(t, u) & t \in \partial C(u)\\
       1 & t \notin C(u)
     \end{cases},
 \end{equation}
and $C(u)$ is a convex set for every $u$. Then, for any $\phi \notin \sC$, there exists $\phi^* \in \sC$ such that
\begin{equation}
  \E_{\theta,\zeta}[\phi^*(T, U)] \geq \E_{\theta,\zeta}[\phi(T, U)], \;\;\text{ for all } (\theta,\zeta)\in \Theta_1.
 \label{eq:inadmissible}
\end{equation}
\label{thm:matthes}
\end{theorem}

Notice that, if \eqref{eq:inadmissible} holds with equality for all $(\theta, \zeta)$, then by the completeness of $(T,U)$ we have $\phi \eqAS \phi^*$. Hence, every admissible test is in $\sC$ or almost surely equal to a test in $\sC$.

In order to apply this result to data splitting, we first introduce a generic exponential family composed of two independent data sets governed by the same parameters:
\begin{model}[Exponential Family with \capSampOrData Splitting]\label{mod:ss}
Model independent random variables $(Y_{1},Y_{2}) \in \cY_1 \times \cY_2$ as
\begin{equation}
  Y_{i} \sim \exp\left\{\theta T_{i}(y) + \zeta ' U_{i}(y) - \psi_{i}(\theta,\zeta)\right\}\; f_{0,i}(y),\quad i=1,2,
\end{equation}
with $\theta\in\R$ and with the models for $Y_{i}$ both satisfying Model~\ref{mod:partExFam}.
\end{model}
Model~\ref{mod:ss} would, for example, cover the case where $Y_1$ and $Y_2$ are the responses for two linear regressions with different design matrices but the same regression coefficients.

For a selection event ${A = A_1 \times \cY_2}$, we say $\phi$ is a
{\em data-splitting test} if $\phi(Y) = \phi_2(Y_2)$; that is, the selection stage uses only $Y_{1}$ and the inference stage uses only $Y_{2}$. Again, by Rao-Blackwellization, we can assume without loss of generality that the test is of the form $\phi(T_2, U_2)$.

Next, define the {\em cutoff gap} $g^*(\phi)$ as the largest $g\geq 0$ for which the acceptance and rejection regions are separated by a ``cushion'' of width $g$. If $T_2^*$ is a conditionally independent copy of $T_2$ given $U_2$, then
\begin{equation}\label{eq:gapEvent}
  g^*(\phi) = \sup \;\{g: \; \P_{\theta,\zeta}( | T_2 - T_2^*| < g,\; \phi(T_2,U_2)>0, \;\phi(T_2^*,U_2)<1) = 0\}.
\end{equation}
Note that the support of $(T_2,T_2^*,U_2)$ does not depend on $\theta$ or $\zeta$; thus, neither does $g^*$. For most tests, $g^*(\phi)=0$. For example, $g^*=0$ if either cutoff is in the interior of $\text{supp}(T_2\gv U_2)$ with positive probability, or if $\phi$ is a randomized test for discrete $(T_2,U_2)$.

Next we prove the main technical result of this section: $\phi$ is inadmissible unless $T_1$ is determined by $U_1$ on $A_1$, within an amount $g^*$ of variability.

\begin{theorem}\label{thm:ssInadm}
Let $T_1^*$ denote a copy of $T_1$ that is conditionally independent given $U_1$ and $Y_1\in A$, and let $\phi$ be a data-splitting test of~(\ref{eq:hypEq}) in Model~\ref{mod:ss}. If
\[
\P_{\theta,\zeta}(|T_1-T_1^*| > g^*(\phi) \gv Y_1\in A) > 0
\]
then $\phi$ is inadmissible.
\end{theorem}

\begin{proof}
  Construct conditionally independent copies $T_i^*$ with $(T_1,T_1^*,U_1) \Perp (T_2,T_2^*,U_2)$, and assume that $\phi$ is of the form $\phi(T, U)$ with $T=T_1+T_2$ and $U=U_1+U_2$ (otherwise we could Rao-Blackwellize it). If $\phi$ is admissible, then by \citet{matthes1967tests}, it must be a.s. equivalent to a test of the form~(\ref{eq:convexU}). That is, there exist $c_i(U)$ for which
  \begin{equation}\label{eq:reqTU}
  \P_{\theta,\zeta}(\phi(T,U) < 1, \; T \notin [c_1(U),c_2(U)] \gv A) =
  \P_{\theta,\zeta}(\phi(T,U) > 0, \; c_1(U) < T < c_2(U)\gv A) = 0.
  \end{equation}

Now, by assumption, there exists $\delta > g^*(\phi)$ for which
\[ B_1 \triangleq \{ |T_1 - T_1^*| > \delta \} \]
occurs with positive probability. By the definition of $g^*(\phi)$ in \eqref{eq:gapEvent}, the event
\[ B_2 \triangleq \{ |T_2 - T_2^*| > \delta,\; \phi(T_2, U_2) > 0,\; \phi(T_2^*, U_2) < 1 \} \]
also occurs with positive probability. Since the two events are independent, $B = B_1 \cap B_2$ occurs with positive probability.

  Next, assume w.l.o.g. that the event in~(\ref{eq:gapEvent}) can occur with $T_2^*>T_2$ (otherwise we could reparameterize with natural parameter $\xi=-\theta$, for which $-T_i$ would be the sufficient statistics for $Y_i$). Then for some $\delta>g^*(\phi)$, the event
  \[
  B = \left\{T_1 + \delta < T_1^*,\;\; T_2 < T_2^* < T_2 + \delta,\;\; \phi(T_2,U_2) > 0,\; \text{ and } \phi(T_2^*,U_2) < 1\right\}
  \]
  occurs with positive probability for all $\theta,\zeta$. On $B$,
  \[
  T_1 + T_2 < T_1 + T_2^* < T_1^* + T_2 < T_1 + T_2^*,
  \]
  but $\phi(T,U) > 0$ for $T=T_1+T_2$ and $T=T_1^*+T_2$ and $\phi(T,U)<1$ for the other two, ruling out the possibility of~(\ref{eq:reqTU}).
\end{proof}

In the typical case $g^*=0$ and we have
\begin{corollary}
  Suppose $\phi$ is a data-splitting test of~(\ref{eq:hypEq}) in Model~\ref{mod:ss} with $g^*(\phi) = 0$. Then $\phi$ is inadmissible unless $T_1$ is a function of $U_1$ on $A$.
\end{corollary}

\begin{example}\label{ex:ssBad}
To illustrate Theorem~\ref{thm:ssInadm}, consider a bivariate version of Example~\ref{ex:univar}:
\begin{equation}
  Y_{i} \sim N(\mu,1),\;\; i=1,2, \;\;\text{ with } Y_{1} \Perp Y_{2},
\end{equation}
in which we condition on the selection event $A=\{Y_{1} > 3\}$.

With \sampOrData splitting, we could construct a 95\% confidence interval using only $Y_{2}$; namely, ${Y_{2} \pm 1.96}$. This interval is valid but does not use all the information available.
A more powerful alternative is to construct an interval based on the law
\begin{equation}
  \cL_{\mu}\left( Y_{1} + Y_{2} \Gv Y_{1} > 3\right),
\end{equation}
which uses the leftover information in $Y_{1}$.

Figure~\ref{fig:ssBad:info} shows the Fisher information that is available to each test as a function of $\mu$. The Fisher information of \sampOrData splitting is exactly 1 no matter what $\mu$ is, whereas the optimal selective test has information approaching 2 as $\mu$ increases. Figure~\ref{fig:ssBad:ci} shows the expected confidence interval length of the equal tailed interval as a function of $\mu$. For $\mu\gg 3$, the \sampOrData splitting interval is roughly 41\% longer than it needs to be (in the limit, the factor is $\sqrt{2}-1$).

Together, the plots tell a consistent story: when the selection event is not too unlikely, discarding the first data set exacts an unnecessary toll on the power of our second-stage procedure.
\end{example}

\begin{figure}
  \centering
  \begin{subfigure}[t]{.4\textwidth}
    \includegraphics[width=\textwidth]{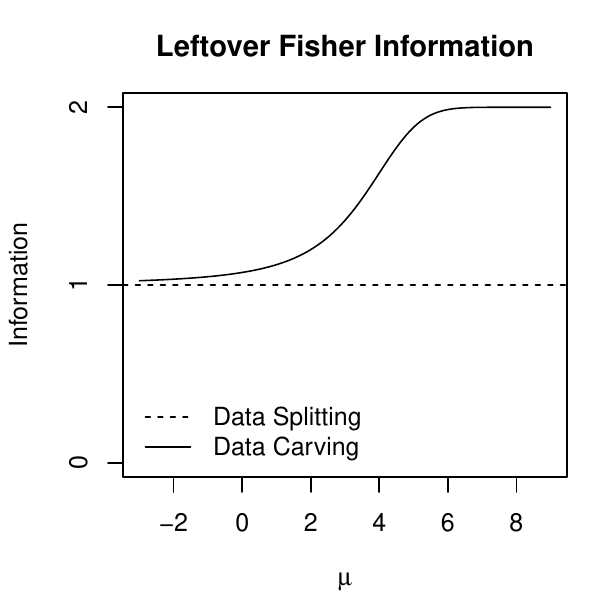}
    \caption{Fisher information available for second-stage inference.}
    \label{fig:ssBad:info}
  \end{subfigure}
  \hspace{.1\textwidth}
  \begin{subfigure}[t]{.4\textwidth}
    \includegraphics[width=\textwidth]{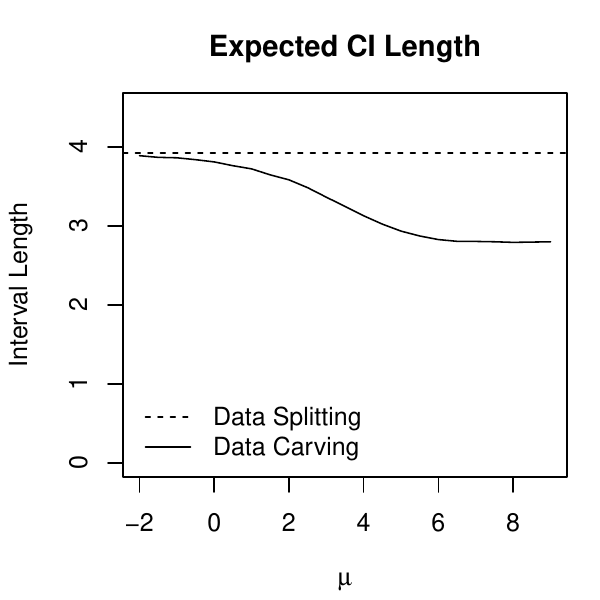}
    \caption{Expected confidence interval length.}
    \label{fig:ssBad:ci}
  \end{subfigure}
  \caption{Contrast between \sampOrData splitting and \sampOrData carving in Example~\ref{ex:ssBad}, in which $Y_{i}\sim N(\mu,1)$ independently for $i=1,2$. \capSampOrData splitting
discards $Y_{1}$ entirely, while \sampOrData carving uses the leftover information in $Y_{1}$ for the second-stage inference. When $\mu\ll 3$, \sampOrData carving also uses about one data point for inference since there is no information left over in $Y_{1}$. But when $\mu\gg 3$, conditioning barely effects the law of $Y_{1}$ and \sampOrData carving has nearly two data points left over.}
  \label{fig:ssBad}
\end{figure}

\section{Selective Inference for Linear Regression}\label{sec:linReg}

For a concrete example of the exponential family framework discussed in Section~\ref{sec:exfam}, we now turn to linear regression, which is one of the most important applications of selective inference. In linear regression, the data arise from a multivariate normal distribution
\begin{equation}\label{eq:satModel}
  Y \sim N_n(\mu, \;\sigma^2 I_n),
\end{equation}
where $\mu$ is modeled as
\begin{equation}\label{eq:subModel}
  \mu = \bX_M \beta^{M}.
\end{equation}
To avoid trivialities, we will assume that $\bX_M$ has full column rank for all $M$ under consideration, so that $\beta^{M}$ is well-defined. 

Depending on whether $\sigma^2$ is assumed known or unknown, hypothesis tests for coordinates $\beta_j^{M}$ generalize either the $z$-test or the $t$-test. In the non-selective case, $z$- and $t$-tests are based on coordinates of the ordinary least squares (OLS) estimator $\hat\beta = \bX_M^\dagger Y$, where  $\bX_M^\dagger$ is the Moore-Penrose pseudoinverse. For a particular $j$ and $M$, it will be convenient to write $\hat\beta_j^M = {\eta_j^M}'Y$ with
\begin{equation}
  {\eta_j^M}= \frac{X_{j\cdot M}}{\|X_{j\cdot M}\|^2}, \quad \text{ where } X_{j\cdot M} = \proj_{\bX_{M\setminus j}}^\perp X_j
\end{equation}
is the remainder after adjusting $X_j$ for the other columns of $X_M$, and $\proj_{\bX_{M\setminus j}}$ denotes projection onto the column space of $\bX_{M\setminus j}$. Letting $\hat\sigma^2=\|\proj_{\bX_M}^\perp Y\|^2/(n-|M|)$, the test statistics
\begin{equation}\label{eq:testStats}
  Z = \frac{{\eta_j^M}'Y}{\sigma \|{\eta_j^M}\|} \quad \text{ and } \quad \widetilde T = \frac{{\eta_j^M}'Y}{\hat\sigma \|{\eta_j^M}\|}
\end{equation}
are respectively distributed as $N(0,1)$ and $t_{n-|M|}$ under $H_0:\, \beta_j^M = 0$. Henceforth, we will suppress the subscript and superscript for $\eta_j^M$, simply writing $\eta$ when there is no ambiguity. The optimal selective $t$- and $z$-tests are based on the same test statistics, but compared against different null distributions.

We consider two distinct modeling frameworks: Section~\ref{sec:reducedModel} concerns inference under the more restrictive {\em selected linear model}, the family of distributions for which~\eqref{eq:subModel} and~\eqref{eq:satModel} both hold, while Section~\ref{sec:fullModel} concerns inference under the more general {\em saturated model} which assumes only~\eqref{eq:satModel} and performs inference on ${\eta_j^M}'\mu$. As we will see, selected-model tests can be more powerful than saturated-model tests, but the extra power comes at a price since the inferences are only valid under more restrictive modeling assumptions. Section~\ref{sec:fullVLinear} compares and contrasts the two approaches.

\subsection{Inference Under the Selected Model}\label{sec:reducedModel}

Suppressing the superscript $M$ in $\beta^{M}$, the selected model has the form
\begin{equation}
  Y \sim \exp\left\{\frac{1}{\sigma^2}\beta'{\bX_M}'y - \frac{1}{2\sigma^2}\|y\|^2 - \psi(\bX_M\beta,\sigma^2)\right\}
\end{equation}
If $\sigma^2$ is known, the sufficient statistics are $X_k'Y$ for $k\in M$, and inference for $\beta_j$ is based on
\begin{equation}
  \cL_{\beta_j}\left(X_j'Y \Gv {\bX_{M\setminus j}}'Y, \; A\right).
\end{equation}
Otherwise, $\|Y\|^2$ represents another sufficient statistic and inference is based on
\begin{equation}
  \cL_{\beta_j/\sigma^2}\left(X_j'Y \Gv {\bX_{M\setminus j}}'Y, \; \|Y\|, \; A\right).
\end{equation}
Decomposing
\begin{align}
  X_j'Y
  &= X_j'\proj_{\bX_{M\setminus j}} Y
  + X_j'\proj_{\bX_{M\setminus j}}^\perp Y \\
  &= X_j'\proj_{\bX_{M\setminus j}} Y + \|X_{j.M}\|^2 \eta'Y,
\end{align}
we see that $Z=\eta'Y / \sigma \|\eta\|$ is a fixed affine transformation of $X_j'Y$ once we condition on $X_{M\setminus j}'Y$. If $\sigma^2$ is known, then, we can equivalently base our selective test on
\begin{equation}
  \cL_{\beta_j}\left(Z \Gv {\bX_{M\setminus j}}'Y, \; A\right).
\end{equation}
While $Z$ is marginally independent of $X_{M\setminus j}'Y$, it is generically not conditionally independent given $A$, so that the null distribution of $Z$ generically depends on $X_{M\setminus j}'Y$.

If $\sigma^2$ is unknown, we may observe further that
\begin{equation}\label{eq:sigmaSq}
\hat \sigma^2 = \frac{\|\proj_{\bX_M}^\perp Y\|^2}{n-|M|}
= \frac{\|Y\|^2 - \|\proj_{\bX_{M\setminus j}} Y\|^2 - (\eta'Y)^2/\|\eta\|^2}{n-|M|}.
\end{equation}
Writing $Z_0(Y) = \eta'Y / \|\eta\|$, we have
$\widetilde T(Y) = (n-|M|) \,Z_0/(\|Y\|^2 - \|\proj_{X_{M\setminus j}} Y\|^2 - Z_0^2)$, which is a monotone function of $\eta'Y$ after fixing $\|Y\|^2$ and $X_{M\setminus j}'Y$. Thus, our test is based on the appropriate conditional law of
\begin{equation}
  \cL_{\beta_j/\sigma^2}\left(\widetilde T \Gv {\bX_{M\setminus j}}'Y, \; \|Y\|, \; A\right).
\end{equation}
Note that, given $A$, $\hat\sigma^2$ in~(\ref{eq:sigmaSq}) is neither unbiased for $\sigma^2$ nor $\chi^2$-distributed. We recommend against viewing it as a serious estimate of $\sigma^2$ in the selective setting.

Constructing a selective $t$-interval is not as straightforward as the general case described in Section~\ref{sec:monte} because $\beta_j$ is not a natural parameter of the selected model; rather, $\beta_j/\sigma^2$ is. Testing $\beta_j=0$ is equivalent to testing $\beta_j/\sigma^2=0$, but testing $\beta_j=b$ for $b\neq 0$ does not correspond to any point null hypothesis about $\beta_j/\sigma^2$. However, we can define
\begin{equation}
  \widetilde Y = Y-b X_j\sim N(\bX \beta - bX_j, \sigma^2I).
\end{equation}
Because $(\beta_j-b)/\sigma^2$ is a natural parameter for $\widetilde Y$, we can carry out a UMPU selective $t$-test for $H_0:\beta_j=b \iff (\beta_j-b)/\sigma^2=0$ based on the law of $\widetilde Y$.

\subsection{Inference Under the Saturated Model}\label{sec:fullModel}

Even if we do not take the linear model~(\ref{eq:subModel}) seriously, there is still a well-defined best linear predictor in the population for design matrix $\bX_M$:
\begin{equation}
  \theta^{M}=
  \argmin_\theta \E_\mu\left[ \|Y - \bX_M\theta\|^2\right]
  = \bX_M^\dagger \mu,
\end{equation}
We call $\theta^{M}$ the {\em least squares coefficients} for $M$. According to this point of view, each $\theta_j^{M}$ corresponds to the linear functional ${\eta_j^M}'\mu$.

This point of view is convenient because the least-squares parameters are well-defined under the more general saturated model~(\ref{eq:fullModel}), leading to meaningful inference even if we do a poor job of selecting predictors. In particular, \citet{berk2013valid} adopt this perspective as a way of avoiding the need to consider multiple candidate statistical models.

Several recent articles have tackled the problem of exact selective inference in linear regression after specific selection procedures~\citep{lee2016exact,loftus2014significance,lee2014marginal}. These works, as well as \citet{berk2013valid}, assume the error variance is known, or that an estimate may be obtained from independent data, and target least-squares parameters in the saturated model.

Under the selected model, $\beta_j^{M} = \theta_j^{M} = \eta'\mu$, whereas under the saturated model $\beta^{M}$ may not exist (i.e., there is no $\beta^M$ such that $\mu = \bX_M\beta^M$). Compared to the selected model, the saturated model has $n-|M|$ additional nuisance parameters corresponding to $\proj_{\bX_M}^\perp \mu$.

We can write the saturated model in exponential family form as
\begin{equation}\label{eq:fullModEF}
  Y \sim \exp\left\{ \frac{1}{\sigma^2}\mu ' y - \frac{1}{2\sigma^2} \|y\|^2 - \psi(\mu,\sigma^2)\right\},
\end{equation}
which has $n+1$ natural paramaters if $\sigma^2$ is unknown and $n$ otherwise. To perform inference on some least-squares coefficient $\theta_j^{M}=\eta'\mu$, we can rewrite~(\ref{eq:fullModEF}) as
\begin{equation}
  Y \sim \exp\left\{ \frac{1}{\sigma^2\|\eta\|^2} \mu'\eta\; \eta'y +
    \frac{1}{\sigma^2} (\proj_{\eta}^\perp\mu)'\; (\proj_\eta^\perp y)
    - \frac{1}{2\sigma^2}\|y\|^2 - \psi(\mu,\sigma^2)\right\}.
\end{equation}

If $\sigma^2$ is known, inference for $\theta_j^{M}$ after selection event $A$ is based on the conditional law $\cL_{\theta_j^{M}}\left( \eta'Y \Gv \proj_\eta^\perp Y, \; A\right)$, or equivalently $\cL_{\theta_j^{M}}\left( Z \Gv \proj_\eta^\perp Y, \; A\right)$.

If $\sigma^2$ is unknown, we must instead base inference on
\begin{equation}\label{eq:lstFullUnk}
  \cL_{\theta_j^{M}/\sigma^2}\left(\eta'Y \Gv \proj_\eta^\perp Y, \; \|Y\|, \; A\right).
\end{equation}

Unfortunately, the conditioning in~(\ref{eq:lstFullUnk}) is too restrictive. The set
\begin{equation}
  \left\{y:\; \proj_\eta^\perp y = \proj_\eta^\perp Y,\;\; \|y\| = \|Y\| \right\}
\end{equation}
is a line intersected with the sphere $\|Y\|S^{n-1}$, and consists only of the two points $\{Y, \; Y - 2\eta'Y\}$, which are equally likely under the hypothesis $\theta_j^{M}=0$. Thus, under the saturated model, conditioning on $\|Y\|$ leaves insufficient information about $\theta_j^{M}$ to carry out a meaningful test.

\subsection{Saturated Model or Selected Model?}\label{sec:fullVLinear}

When $\sigma^2$ is known, we have a choice whether to carry out the $z$-test with test statistic $Z = \eta'Y/\sigma \|\eta\|$ in the saturated or the selected model. In other words, we must choose either to assume that $\proj_{\bX_M}^\perp \mu = 0$ or to treat it as an unknown nuisance parameter. Writing
\begin{equation}
  U = {\bX_{M\setminus j}}' Y, \quad \text{ and } \quad V = \proj_{\bX_M}^\perp Y,
\end{equation}
we must choose whether to condition on $U$ and $V$ (saturated model) or only $U$ (selected model). Conditioning on both $U$ and $V$ can never increase our power relative to conditioning only on $U$, and (unless the tests coincide) will lead to an inadmissible test per Theorem~\ref{thm:ssInadm}.

In the non-selective case, this choice makes no difference at all since $T, U,$ and $V$ are mutually independent.
In the selective case, however, the choice may be of major consequence as it can lead to very different tests. In general, $T, U,$ and $V$ are not conditionally independent given $A$, and $\proj_{\bX_M}^\perp \mu$ may play an important role in determining the conditional distribution of $T$. If we needlessly condition on $V$, we may lose a great deal of power, whereas failing to condition on $V$ could lead us astray if $\proj_{\bX_M}^\perp \mu$ is large. A simple example can elucidate this contrast.

\begin{example}\label{ex:bivar}
  Suppose that $y\sim N_2(\mu, I_2)$, with design matrix $\bX = I_2$, and we choose the best-fitting one-sparse model. That is, we choose $M=\{1\}$ if $|Y_1|>|Y_2|$, and $M=\{2\}$ otherwise.

Figure~\ref{fig:fullvred} shows one realization of this process with $Y=(2.9,2.5)$. $|Y_1|$ is a little larger than $|Y_2|$, so we choose $M=\{1\}$. The yellow highlighted region $A=\{|Y_1|>|Y_2|\}$ is the chosen selection event, and the selected model is
\begin{equation}
  Y \sim N_2\left((\mu_1,0), I_2\right).
\end{equation}
In this case, $T=Y_1$, $V=Y_2$, and there is no $U$ since $\bX_M$ has only one column. The selected-model test is based on $\cL(Y_1 \gv A)$, whereas the saturated-model test is based on $\cL(Y_1 \gv Y_2, A)$. The second conditioning set, a union of two rays, is plotted in brown. Under the hypothesis $\mu=0$, the realized $|Y_1|$ is quite large given $A$, giving $p$-value 0.007. By contrast, $|Y_1|$ is not terribly large given $\{Y_2=2.5\}\cap A = \{Y_2=2.5, |Y_1|>2.5\}$, leading to $p$-value 0.30.
\end{example}

The selected-model approach is especially well-suited for testing goodness of fit of a selected linear model --- in that case, we prefer the test {\em not} to have level $\alpha$, but rather to reject with high probability, when important variables are not selected. \citet{fithian2015adaptive} consider sequential goodness-of-fit testing in a ``path'' of increasingly complex models selected by a method like the lasso or forward stepwise regression. As Example~\ref{ex:bivar} illustrates, the saturated-model $p$-value is especially large for ``near ties'' when $|Y_1|$ is not much larger than $|Y_2|$. As a result, the selected-model test can be much more powerful in early steps of the path where multiple strong variables compete to enter the model first. For more details see \citet{fithian2015adaptive}.

\begin{figure}
  \centering
  \begin{subfigure}[t]{.4\textwidth}
    \includegraphics[width=\textwidth]{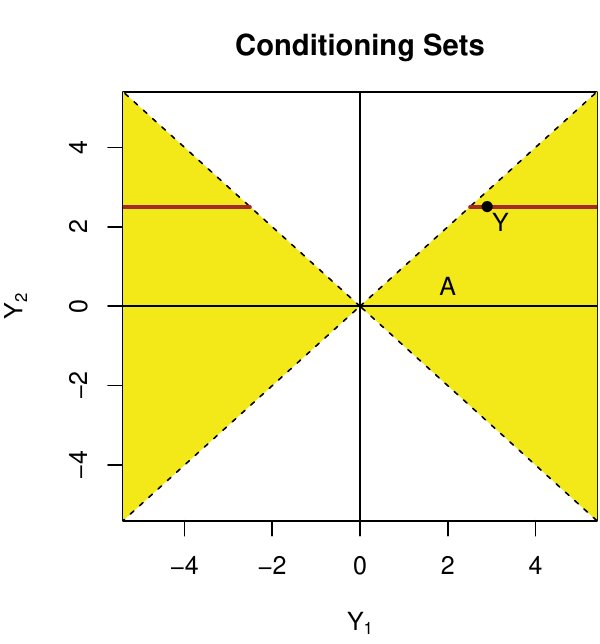}
    \caption{For $Y=(2.9,2.5)$, the selected-model conditioning set is
      $A=\{y:\;|y_1|>|y_2|\}$, a union of quadrants,
      plotted in yellow. The saturated-model conditioning set
      is ${\{y:\; y_2=2.5\}\cap A} = {\{y:\;y_2=2.5, |y_1|>2.5\}}$,
      a union of rays, plotted in brown.}
    \label{fig:fullvredSets}
  \end{subfigure}
  \hspace{.1\textwidth}
  \begin{subfigure}[t]{.4\textwidth}
    \includegraphics[width=\textwidth]{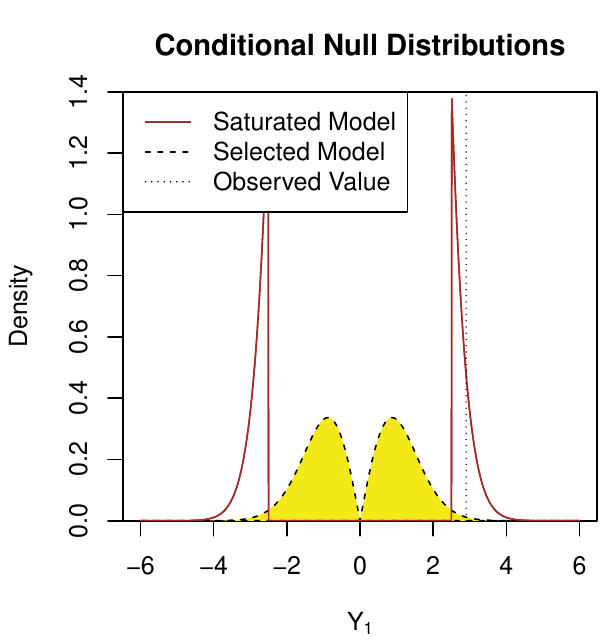}
    \caption{ Conditional distributions of $Y_1$ under
      $H_0:\mu_1 = 0$. Under the hypothesis
      $\mu=0$, the realized  $|Y_1|$ is  quite large given $A$,
      giving  $p$-value 0.007. By contrast, $|Y_1|$ is not too large
      given $A \cap \{y:\; y_2=Y_2\}$, giving
      $p$-value 0.3.}
  \end{subfigure}
  \caption{Contrast between the saturated-model and selected-model tests
    in Example~\ref{ex:bivar}, in which we fit a one-sparse model with
    design matrix $\bX=I_2$. The selected-model test is based
    on  $\cL_0(Y_1 \gv A)$, whereas the saturated-model test is based
    on $\cL_0(Y_1  \gv  Y_2, A)$.}
  \label{fig:fullvred}
\end{figure}

\section{Computations}\label{sec:approxInt}

We saw in Section~\ref{sec:exfam} that inference in the one-parameter exponential family requires knowing the conditional law $\cL_\theta(T\gv U, A)$. In a few cases, such as in the saturated model viewpoint, this conditional law can be determined fairly explicitly. In other cases, we will need to resort to Monte Carlo sampling. In this section, we suggest some general strategies.

\subsection{Gaussians Under the Saturated Model}\label{sec:classifying}

As we discussed in Section~\ref{sec:fullModel}, the previous papers by \citet{lee2016exact,loftus2014significance,lee2014marginal} adopted the saturated model viewpoint with known $\sigma^2$. In this case, $\cL_\theta(T\gv U, A) = \cL_\theta\left( \eta'Y \Gv \proj_\eta^\perp Y, \; A\right)$ is a truncated univariate Gaussian, since $\eta' Y$ is a Gaussian random variable and $\proj_\eta^\perp Y$ is independent of $\eta' Y$. If $A$ is convex, then the truncation is to an interval $[\cV^-(Y), \cV^+(Y)]$, where the endpoints represent the maximal extent one can move in the $\eta$ direction at a ``height'' of $\proj_\eta^\perp Y$, while still remaining inside $A$, i.e.,
\begin{align}
  \cV^{+}(Y) &= \sup_{\{t:\,Y + t\eta \,\in\, A\}} \eta'(Y+t\eta)\\[5pt]
  \cV^{-}(Y) &= \inf_{\{t:\,Y + t\eta \,\in\, A\}} \eta'(Y+t\eta).
\end{align}
The geometric intuition is illustrated in Figure~\ref{fig:blob}.

When $A$ is specifically a polytope, we can obtain closed-form expressions for $\cV^-$ and $\cV^+$. The generalization to regions $A$ that are non-convex is straightforward (i.e., instead of truncating to a single interval, we truncate to a union of intervals). For further discussion of these points, see \citet{lee2016exact}.

 \begin{figure}
   \centering
 \def\svgwidth{0.4\textwidth}
\input{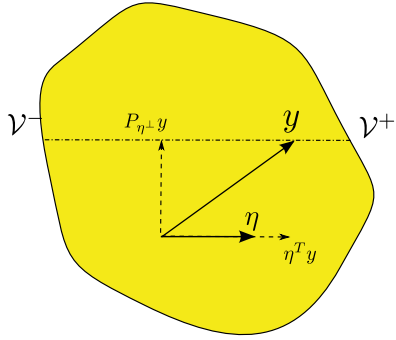}
   \caption{Saturated-model inference for a generic convex selection set for ${Y\sim N(\mu, I_n)}$. After conditioning on the yellow set $A$, $\cV^+$ is the largest $\eta'Y$ can get while $\cV^-$ is the smallest it can get. Under $H_0:\;\eta'\mu = 0$, the test statistic $\eta'Y$ takes on the distribution of a standard Gaussian random variable truncated to the interval $[\cV^-, \cV^+]$. As a result, $W(Y) = \frac{ \Phi(\eta'Y) - \Phi(\cV^-)}{ \Phi(\cV^+) - \Phi(\cV^-)}$ is uniformly distributed.}
   \label{fig:blob}
 \end{figure}

\subsection{Monte Carlo Tests and Intervals}\label{sec:monte}

In a more generic setting, we may not have an easy formula for conditional law of $T$. In that case, there are several options for inference using Monte Carlo methods.

If we can obtain a stream of samples from $\cL_\theta(T\gv U, A)$ for any value of $\theta$, then we can carry out hypothesis tests and construct intervals. This can be done efficiently via rejection sampling if, for example, we can sample efficiently from $\cL_\theta(Y\gv U)$ and $\P_\theta(Y \in A \gv U)$ is not too small. Otherwise, more specialized sampling approaches may be required. A little more abstractly, we now consider constructing a test based on the statistic $Z$, which is distributed according to a one-parameter exponential family
\begin{equation}
  Z \sim g_\theta(z) = e^{\theta z - \psi(\theta)}\,g_0(z).
\end{equation}

\paragraph{Exact Monte Carlo Tests}
Suppose that, in addition to $Z$, we are given an independent sequence
from the reference distribution
\begin{equation}
  Z_1, \ldots, Z_n \simiid g_0(z).
\end{equation}
Then an exact Monte Carlo one-sided test of
$H_0:\,\theta \leq 0$ rejects if the observed value $Z$ is
among the $(n+1)\,\alpha$ largest of $Z,Z_1,\ldots,Z_n$ \citep{barnard1963discussion}.

Even if i.i.d. samples are not available, the same procedure has level $\alpha$ provided the law of $(Z,Z_1,\ldots,Z_n)$ is {\em exchangeable} under $H_0$. \citet{besag1989generalized} propose an ingenious procedure for obtaining such an exchangeable sequence when we only know how to run a Markov chain with stationary distribution $g_0$: Beginning at $Z$, take $k\geq 1$ steps backward in the chain to $\widetilde Z$. Then, run $n$ independent chains $k$ steps forward, beginning each chain at $\widetilde Z$, and letting $Z_i$ denote the end state of the $i$th chain. If $Z\sim g_0$, the sequence is exchangeable. 

Note that, while this test has level $\alpha$ for {\em any} $k, n \geq 0$, using small values of $k,n$ makes the test more random, reducing its power. If the chain is irreducible, then as $k,n\to \infty$, the test converges to the deterministic right-tailed level-$\alpha$ test of $H_0$.

\paragraph{Approximate Monte Carlo Intervals}

By reweighting the samples, we can use $(Z_1,\ldots,Z_n)$ to test $H_0:\, \theta\leq \theta_0$ for any other $\theta_0$. Denote the importance-weighted empirical expectation as
\begin{align}
  \widehat \E_\theta \,h(Z) &=
  \frac{\sum_{i=1}^n h(Z_i) e^{\theta Z_i}}
  {\sum_{i=1}^n e^{\theta Z_i}}\\
  &\toAS \E_\theta \,h(Z) \quad \text{ as $n\to\infty$ for integrable } h.
\end{align}
In effect, we have put an exponential family ``through'' the empirical distribution of the $Z_i$ in the manner of \citet{efron1996using}; see also \citet{besag2001markov}. The Monte Carlo one-sided cutoff for a test of $H_0:\, \theta\leq \theta_0$ is the smallest $c_2$ for which $\widehat \P_{\theta_0} (Z> c_2) \leq \alpha$.
The test rejects for $Z>c_2$ and randomizes appropriately at $Z=c_2$.

The two-sided test of $H_0:\,
\theta=\theta_0$ is a bit more involved, but similar in principle. We can solve for
$c_1,\gamma_1,c_2,\gamma_2$ for which
\begin{align}\label{eq:levelAlpha}
  \widehat \E_{\theta_0} \phi(Z) &= \alpha\\\label{eq:uncorr}
  \widehat \E_{\theta_0} \left[Z\phi(Z) \right]
  &= \alpha\,\widehat\E_{\theta_0} Z.
\end{align}
In Appendix~\ref{sec:mcUMPU} we discuss how
(\ref{eq:levelAlpha}--\ref{eq:uncorr}) can be solved
efficiently for fixed $\theta_0$ and inverted to obtain a confidence
interval. Monte Carlo inference as described above is computationally
straightforward once $Z_1,\ldots,Z_n$ are obtained.

More generally, the $Z_i$ could represent importance samples with weights $W_i$, or steps in a Markov chain with stationary distribution
$g_0(z)$. The same methods apply as long as we still have
\begin{align}
  \widehat \E_\theta\, h(Z) &=
  \frac{\sum_{i=1}^n W_i\, h(Z_i)\, e^{\theta Z_i}}
  {\sum_{i=1}^n W_i\, e^{\theta Z_i}}\\
  &\toAS \E_\theta\, h(Z), \quad \text{ for integrable } h.
\end{align}

Numerical problems may arise in solving (\ref{eq:levelAlpha}--\ref{eq:uncorr}) for $\theta_0$ far away from the reference parameter used for sampling. Combining appropriately weighted samples from several different reference values can help to keep the effective sample size from getting too small for any $\theta_0$. For further references on Monte Carlo inference see
\citet{jockel1986finite,forster1996monte,mehta2000efficient}.

\subsection{Sampling Gaussians with Affine and Quadratic Constraints}\label{sec:affine}

In the case where $Y$ is Gaussian, several simplifications are possible. For one, there are many ways to sample from a truncated multivariate Gaussian distribution. In this paper, we use hit-and-run Gibbs sampling algorithms, while \citet{pakman2014exact} suggest another approach based on Hamiltonian Monte Carlo.

Efficient sampling from multivariate Gaussian distributions under such constraints is the main algorithmic challenge for most of the Gaussian selective tests proposed in this paper. The works cited above use the saturated model exclusively which means they do not require any sampling.

In many cases, the sampling problem may be greatly facilitated by refining the selection variable that we use. For example, \citet{lee2016exact} propose conditioning on the variables selected by the lasso as well as the signs of the fitted $\hat\beta_j$, leading to a selection event consisting of a single polytope in $\R^n$. If we condition only on the selected variables and not on the signs, the selection event is a union of up to $2^s$ polytopes, where $s$ is the number of variables in the selected model (though most of the polytopes might be excluded after conditioning on $U$).

Refining the selection variable never impairs the selective validity of the procedure, but it typically leads to a loss in power. However, this loss of power may be quite small if, for example, the conditional law puts nearly all of its mass on the realized polytope. This price in power is acceptable if it is the only way to obtain a tractable test. Quantifying the tradeoff between computation and power is an interesting topic for further work.

When carrying out selective $t$-tests, it is necessary to condition further on the realized vector length $\|Y\|$, adding a quadratic equality constraint to the support. To deal with this, we sample instead from a ball and project the samples onto the sphere using an importance sampling scheme. Appendix~\ref{sec:sphereSample} gives details.

\section{Selective Inference in Non-Gaussian Settings}\label{sec:examples}

In this section we describe tests in two simple non-Gaussian settings, selective inference in a binomial problem, and tests involving a scan statistic in Poisson process models. More generally, we address the question of selective inference in generalized linear models.

\subsection{Selective Clinical Trial}\label{sec:clinical}

To illustrate the application of our approach in a simple non-Gaussian setting we discuss a selective clinical trial with binomial data. The experiment discussed here is similar to an adaptive design proposed by~\citet{sill2009drop}.

Consider a clinical trial with $m$ candidate treatments for heart disease. We give treatment $j$ to $n_j$ patients for $0 \leq j \leq m$, with $j=0$ corresponding to the placebo. The number of patients on treatment $j$ to suffer a heart attack during the trial is
\begin{equation}
  Y_j \simind \text{Binom}(p_j, n_j), \quad \text{ with }   \log \frac{p_j}{1-p_j} = \left\{\begin{matrix} \theta & j = 0\\ \theta - \beta_j & j > 0\end{matrix}\right. ,
\end{equation}
so $\beta_j$ measures the efficacy of treatment $j$. The likelihood for $Y$ is
\begin{equation}
  Y \sim \exp\left\{\theta \sum_{j=0}^m y_j - \sum_{j=1}^m \beta_j y_j
    - \psi(\theta, \beta) \right\}\;\prod_{j=0}^m\binom{Y_j}{n_j},
\end{equation}
an exponential family with $m+1$ sufficient statistics. Define $\hat p_j = Y_j/n_j$, and let $\hat p_{(j)}$ denote the $j$th smallest order statistic.

After observing the data, we select the best $k<m$ treatments in-sample, then construct a confidence interval for each one's odds ratio relative to placebo. If there are ties, we select all treatments for which $\hat p_j \leq \hat p_{(k)}$ (so that we could possibly select more than $k$ treatments).

For simplicity, assume that treatments $1,\ldots,k$ are the ones selected. Inference for $\beta_1$ is then based on the conditional law
\begin{equation}
  \cL_{\beta_1}\left( Y_1  \Gv
    \sum_{j=0}^m Y_j,\;\, Y_2, \ldots, Y_J, \;
     \{j=1 \text{ selected}\}\right)
\end{equation}
Under this law, $Y_2,\ldots,Y_m$ are fixed, as is $Y_0+Y_1$, with $Y_0$ and $Y_1$ the only remaining unknowns. Before conditioning on selection, we have the two-by-two multinomial table

\renewcommand{\arraystretch}{2}

\begin{center}
\begin{tabular}{ r|c|c| }
\multicolumn{1}{r}{}
 &  \multicolumn{1}{c}{Control}
 & \multicolumn{1}{c}{Treatment} \\
\cline{2-3}
Heart attack & $Y_0$ & $Y_1$ \\
\cline{2-3}
No heart attack & $n_0-Y_0$ & $n_1-Y_1$ \\
\cline{2-3}
\end{tabular}
\end{center}

The margins are fixed, and conditioning on selection gives an additional constraint that $Y_1 \leq n_1 \hat p_{(k)}$, where the right-hand side is known after conditioning on the other $Y_j$. Rejecting for conditionally extreme $Y_1$ amounts to a selective Fisher's exact test. Aside from the constraint on its support, the distribution of $Y_1$ is hypergeometric if $\beta_1=0$ and otherwise noncentral hypergeometric with noncentrality parameter $\beta_1$. We can use this family to construct an interval for $\beta_1$.

\subsection{Poisson Scan Statistic}\label{sec:poisson}

As a second simple example, consider observing a Poisson process $Y=\{Y_1,\ldots,Y_{N(Y)}\}$ on the interval $[0,1]$ with piecewise-constant intensity, possibly elevated in some unknown window $[a,b]$. That is, $Y\sim \text{Poisson}(\lambda(t))$ with
\begin{equation}
  \lambda(t) = \begin{cases}
    e^{\alpha+\beta} &  t\in[a,b]\\
    e^\alpha &  \text{otherwise.}
  \end{cases}
\end{equation}
Our goal is to locate $[a,b]$ by maximizing some scan statistic, then test whether $\beta>0$ or construct a confidence interval for it. Assume we always have $[\hat a, \hat b] = [Y_i,Y_j]$ for some $i,j$; this is true, for example, if we use the multi-scale-adjusted likelihood ratio statistic proposed in \citet{rivera2013optimal}.

The density of $Y$ can be written in exponential family form as
\begin{align}
  Y
  &\sim \exp\left\{ \sum_{i=1}^{N(y)} \log \lambda(y_i)
    - \int_0^1 \lambda(s)\,ds\right\} \\[3pt]
  & = \exp\big\{\alpha\, N(Y) + \beta \,T(y) - \psi(\alpha,\beta)\big\},
\end{align}
where
\begin{equation}
  T(y) = \sum_{i=1}^{N(y)} \1\{y_i\in [a,b]\} \quad\text{ and }\quad \psi(\alpha,\beta) = e^\alpha(1-b+a) + e^{\alpha+\beta}(b-a).
\end{equation}
If $A$ is the event that $[a,b]$ is chosen, we carry out inference with respect to $\cL_\beta\left(T \gv N, A\right)$. Note that under $\beta=0$ and conditional on $N$, $Y$ is an i.i.d. uniform random sample on $[0,1]$.

Once we condition on the event $\{a,b\in Y\}$, the other $N-2$ values are uniform. Thus, we can sample from $\cL_\beta\left(T \gv N, A\right)$ with $\beta=0$ by taking $Y$ to include $a, b$, and $N-2$ uniformly random points, then rejecting samples for which $[a,b]$ is not the selected window.

\subsection{Generalized Linear Models}\label{sec:glm}

Our framework extends to logistic regression, Poisson regression, or other generalized linear model (GLM) with response $Y$ and design matrix $\bX$, since the GLM model may be represented as an exponential family of the form
\begin{align}
  Y &\sim \exp\left\{ \beta'\;\bX'y -
    \psi(\bX\beta)\right\}\; f_0(y).
\end{align}

As a result, we can proceed just as we did in the case of linear regression in the reduced model, conditioning on $U={\bX_{M\setminus j}}'Y$ and basing inference on $\cL_{\beta_j^{M}}(X_j'Y \gv U, A)$.

A difficulty may arise for logistic or Poisson regression due to the discreteness of the response distribution $Y$. If some control variable $X_1$ is continuous, then for almost every realization of $X$, all configurations of $Y$ yield unique values of $U=X_1'Y$. In that case, conditioning on $X_1'Y$ means conditioning on $Y$ itself. No information is left over for inference, so that the best (and only) exact level-$\alpha$ selective test is the trivial one $\phi(Y)\equiv\alpha$. By contrast, if all of the control variables are discrete variables like gender or ethnicity, then conditioning on $U$ may not constrain $Y$ too much.

Because $\bX'Y$ is approximately a multivariate Gaussian random variable, a more promising approach may be to base inference on the asymptotic Gaussian approximation as in \citet{taylor2016post}.

\section{Simulation: High-Dimensional Regression}\label{sec:simulation}

As a simple illustration, we compare selective inference in linear regression after the lasso for $n=100,p=200$. Here, the rows of the design matrix $\bX$ are drawn from an equicorrelated multivariate Gaussian distribution with pairwise correlation $\rho=0.3$ between the variables. The columns are normalized to
have length 1.

We simulate from the model
\begin{equation}
  Y \sim N(\bX \beta, I_n),
\end{equation}
with $\beta$ 7-sparse and its non-zero entries set to $7$. The magnitude of $\beta$ was chosen so that
\sampOrData splitting with half the data yielded a superset of the true variables on roughly 20\% of instances.
For \sampOrData splitting and carving,  $Y$ is partitioned into selection and inference data sets $Y_{1}$ and $Y_{2}$, containing $n_1$ and $n_2=n-n_1$ data points respectively.

We assume the error variance is known and carry out the Lasso on $Y_1$ with Lagrange parameter
$$\lambda=2\E(\|X^T\epsilon\|_{\infty}), \quad \epsilon \sim N(0, I_n)$$ as described in \citep{negahban_unified}. We then compare two post-selection inference procedures:
\begin{description}
\item[\capSampOrData Splitting after Lasso on $Y_{1}$ ($\text{Split}_{n_1}$):] Use the lasso on $Y_{1}$ to select the model, and use $Y_{2}$ for inference.
\item[\capSampOrData Carving after Lasso on $Y_{1}$ ($\text{Carve}_{n_1}$):] Use the lasso on $Y_{1}$ to select the model, and use $Y_{2}$ and whatever is left over of $Y_{1}$ for inference.
\end{description}
For the \sampOrData carving procedures, we use the selected-model $z$-test of Section~\ref{sec:reducedModel}.
In addition, we condition on the signs of the active lasso coefficients, so procedure $\text{Carve}_{100}$ is the inference-after-lasso test proposed in \citet{lee2016exact}.\footnote{Because of the form of the selection event when we use the lasso after $n$ data points, the test statistic is conditionally independent of $\proj_{\bX_{M}}^\perp Y$. Thus, there is no distinction between the saturated- and selected-model $z$-tests after the lasso on all $n$ data points.}

We know from Theorem~\ref{thm:ssInadm} that procedure $\text{Carve}_{n_1}$ strictly dominates procedure $\text{Split}_{n_1}$ for any $n_1$, but there is a selection--inference tradeoff between data-carving procedures $\text{Carve}_n$ and $\text{Carve}_{n_1}$ for $n_1<n$. $\text{Carve}_{n}$ uses all of the data for selection, and is therefore likely to select a superior model, whereas procedure $\text{Carve}_{n_1}$ reserves more power for the second stage.

Let $R$ be the size of the model selected and $V$ the number of noise variables included. We compare the procedures with respect to aspects of their selection performance:
\begin{itemize}
  \item chance of screening, i.e. obtaining a correct model ($\P(R-V=7)$ or $p_{\text{screen}}$).
  \item expected number of noise variables selected ($\E[V]$),
  \item expected number of true variables selected ($\E[R-V]$),
  \item false discovery rate of true variables selected
    ($\E[V/\max(R, 1)]$ or FDR),
\end{itemize}

Conditional on having obtained a correct model, we also compare them on aspects of their second stage performance:
\begin{itemize}
\item probability of correctly rejecting the null for one of the true variables (Power),
\item probability of incorrectly rejecting the null for a
  noise variable (Level).
\end{itemize}

\begin{table}
  \centering
  \begin{tabular}{|l|c|c|c|c|c|c|c|}
\hline
            Algorithm &  $p_{\text{screen}}$ &  $\mathbb{E}[V]$ &  $\mathbb{E}[R-V]$ &  FDR &  Power &  Level \\
\hline
 $\text{Carve}_{100}$ &                 0.99 &             8.13 &               6.99 & 0.54 &   0.80 &   0.05 \\
  $\text{Split}_{50}$ &                 0.09 &             9.13 &               4.74 & 0.66 &   0.93 &   0.06 \\
  $\text{Carve}_{50}$ &                 0.09 &             9.13 &               4.74 & 0.66 &   0.99 &   0.06 \\
  $\text{Split}_{75}$ &                 0.68 &             9.24 &               6.59 & 0.58 &   0.47 &   0.05 \\
  $\text{Carve}_{75}$ &                 0.68 &             9.24 &               6.59 & 0.58 &   0.97 &   0.06 \\
\hline
\end{tabular}

\caption{Simulation results. $p_{\text{screen}}$ is the probability of successfully selecting all 7 true variables, and Power is the power, conditional on successful screening, of tests on the true variables. The more data we use for selection, the better the selected model's quality is, but there is a cost in second-stage power. $\text{Carve}_{75}$ appears to be finding a good tradeoff between these competing goals. $\text{Carve}_{n_1}$ always outperforms $\text{Split}_{n_1}$, as predicted by Theorem~\ref{thm:ssInadm}.}
\label{tab:results}
\end{table}

\begin{table}
\centering
\begin{tabular}{|l|c|c|c|c|c|c|c|}
\hline
            Algorithm &  $p_{\text{screen}}$ &  $\mathbb{E}[V]$ &  $\mathbb{E}[R-V]$ &  FDR &  Power &  Level \\
\hline
 $\text{Carve}_{100}$ &                 0.97 &             8.11 &               6.97 & 0.54 &   0.80 &   0.04 \\
  $\text{Split}_{50}$ &                 0.09 &             9.20 &               4.77 & 0.66 &   0.93 &   0.05 \\
  $\text{Carve}_{50}$ &                 0.09 &             9.20 &               4.77 & 0.66 &   0.99 &   0.06 \\
\hline
\end{tabular}

  \caption{Simulation results under misspecification. Here, errors $\epsilon$ are drawn independently from Student's $t_5$. Our conclusions are identical to Table~\ref{tab:results}.}
\label{table:T5}
\end{table}

\begin{figure}
  \centering
  \begin{subfigure}[t]{.4\textwidth}
    \includegraphics[width=\textwidth]{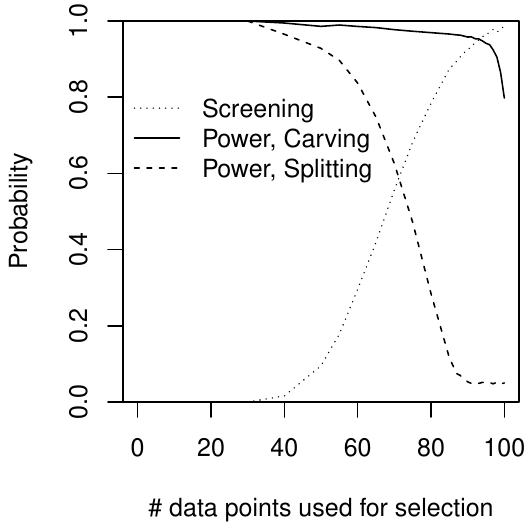}
    \caption{Probability of successful screening, and power conditional on screening, for $\text{Split}_{n_1}$ and $\text{Carve}_{n_1}$.}
    \label{fig:lassoTRADE1}
  \end{subfigure}
  \hspace{.1\textwidth}
  \begin{subfigure}[t]{.4\textwidth}
    \includegraphics[width=\textwidth]{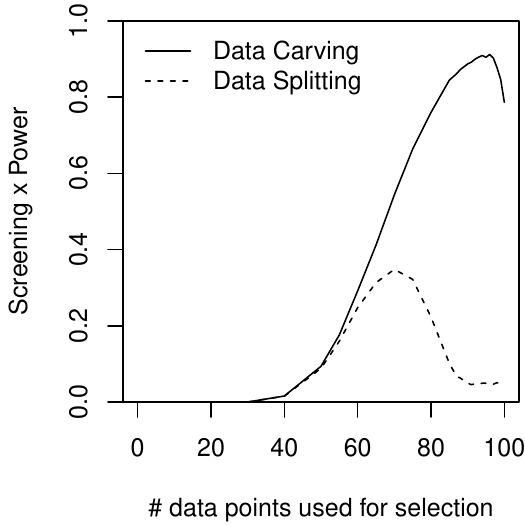}
    \caption{Probability of successful screening times power conditional on screening, for $\text{Split}_{n_1}$ and $\text{Carve}_{n_1}$.}
    \label{fig:lassoTRADE2}
  \end{subfigure}
  \caption{Tradeoff between power and model selection.
As $n_1$ increases and more data is used in the first stage, we have a better chance of successful screening (picking all the true nonzero variables). However, increasing $n_1$ also leads to reduced power in the second stage. Data splitting suffers much more than data carving, though both are affected.}
  \label{fig:lassoTRADEOFF}
\end{figure}

The results, shown in Table~\ref{tab:results}, bear out the intuition of Section~\ref{sec:admissibility}.
Because procedure $\text{Carve}_{100}$ uses the most information in the first stage, it performs best in terms of model selection, but pays a price in lower second-stage power relative to $\text{Split}_{50}$ or $\text{Carve}_{50}$. The procedure $\text{Carve}_{50}$ clearly dominates $\text{Split}_{50}$, as expected. Increasing $n_1$ from 50 to 75 improves $p_{\text{screen}}$ for $\text{Split}_{75}$, but $\text{Split}_{75}$ suffers a drop in power. Procedure $\text{Carve}_{75}$ seems to strike a better compromise.

Figure~\ref{fig:lassoTRADEOFF} shows the tradeoff curve of model selection success (as measured by the probability of successful screening) against second-stage power conditional on successful screening. As $n_1$ increases, stage-one performance improves while stage-two performance declines, but the decline is much slower for \sampOrData carving. Surprisingly, $\text{Carve}_{98}$ and $\text{Carve}_{99}$ have much higher power than $\text{Carve}_{100}$: 91\%, 86\%, and 80\% respectively. We cannot explain why holding out just one or two data points in the first stage improves power so dramatically. Better understanding this tradeoff is an interesting topic of further work.

Finally, to check the robustness of \sampOrData carving, we replace the Gaussian errors with independent errors drawn from Student's $t$ distribution with five degrees of freedom. The numbers barely change at all; see Table~\ref{table:T5}. \citet{tian2017asymptotics} rigorously analyze the case of non-Gaussian errors.

\section{Conditioning as a Device for Multiple Inference}\label{sec:multiple}

To this point we have argued for controlling selective type I error as a goal in its own right, but it can also serve as a device for controlling more traditional multiple inference goals. In this section we discuss two examples: confidence intervals for selected parameters that control the false coverage-statement rate (FCR) and familywise error rate (FWER).

Suppose that $\theta_q,\;q=1,\ldots,m$ correspond to parameters of a common (fixed) model $M$. We adaptively designate a number $R(Y)=|\hcQ(Y)|$ of them as interesting and construct a confidence interval $C_q(Y)$ for each $q\in \hcQ$. \citet{benjamini2005false} propose controlling the {\em false coverage-statement rate} (FCR)
\begin{equation}\label{eq:fcrV}
  \E\left[\frac{V}{\max(R,1)}\right],\quad \text{ where }\quad
  V(Y) = \left|\left\{q:\, q\in\hcQ, \;\theta_q(F) \notin C_q(Y)\right\}\right|
\end{equation}
is the number of non-covering intervals constructed. 

Other authors have addressed inference after selection by proposing to control the FWER, the chance that any selected test incorrectly rejects the null or any constructed confidence interval fails to cover its parameter. For example, the ``post-selection inference'' (PoSI) method of \citet{berk2013valid} constructs simultaneous $(1-\alpha)$ confidence intervals for the least-squares parameters of all linear regression models that were ever under consideration. As a result, no matter how we choose the model, the overall probability of constructing any non-covering interval is controlled at $\alpha$.

By choosing appropriate selection variables $S_q$, we can control the FCR or FWER as desired using intervals with selective coverage. Our proof generalizes and extends a result in~\citet{weinstein2013selection}, who also use conditional control to achieve FCR control in a specialized setting. Using a similar proof, we also show that using an {\em adaptive Bonferroni rule}, which adjusts the test's level based on the (random) number of intervals actually constructed, can achieve FWER control.

\begin{proposition}[FCR and FWER Control via Selective Error Control]\label{prop:fcrCont}
  Assume $\cQ$ is countable with each $q\in \cQ$ corresponding to a different parameter $\theta_q$ for the same model $M$. Let $R(Y) = |\hcQ(Y)|$ with $R(Y)<\infty$ a.s., and define $V(Y)$ as in~(\ref{eq:fcrV}).

If each $C_q$ enjoys coverage at level $1-\alpha$ given $S_q = \left(\1_{A_q}(Y),R(Y)\right)$, then the collection of intervals $(C_q, q\in \hcQ)$ controls the FCR at level $\alpha$:
\begin{equation}
  \E\left[\frac{V}{\max(R,1)}\right] \;\;\leq\;\;
  \E\left[\frac{V}{R} \Gv R \geq 1\right] \;\;\leq\;\; \alpha.
\end{equation}
If each $C_q$ enjoys coverage at level $1-\alpha/R(Y)$ given $S_q$, then $(C_q, q\in \hcQ)$ controls the FWER at level $\alpha$:
\begin{equation}
  \P\left[V \geq 1\right] \;\;\leq\;\; \alpha.
\end{equation}

\end{proposition}
\begin{proof}
  Let $V_q(Y) = \1\left\{q\in \hcQ(Y),\, \theta_q(F) \notin C_q(Y)\right\}$,
  so that $V = \sum_{q\in\cQ}V_q$. If $C_q$ has level-$\alpha$ selective coverage, then for $R\geq 1$, and for any $F\in M$,
  \begin{equation}
    \E_F\left[V \gv R \right]
    \;\;=\;\; \sum_{q\in\cQ} \E_F\left[ V_q \gv R\right]
    \;\;\leq\;\; \sum_{q\in\cQ}
    \alpha\,\E_F\left[ \1_{A_q}(Y) \gv R\right]
    \;\;=\;\; \alpha R,
  \end{equation}
  hence $\E\left[V/R \gv R\right] = \alpha$ for each $R\geq 1$.

  We can repeat the argument when $C_q$ has level-$\alpha/R$ selective coverage, we obtain
  $\E_F\left[V \gv R \right] \leq \alpha$. Marginalizing each bound over $R(Y)$ gives the result.
\end{proof}

However, the converse of Proposition~\ref{prop:fcrCont} is not true: FWER control does {\em not} in general guarantee control of relevant selective error rates. For example, suppose that we construct an interval for the effect of red meat consumption on heart disease ($Q(Y)=1$) with probability 0.9 and for the effect of statins on heart disease ($Q(Y)=2$) otherwise. If $C_1$ and $C_2$ have selective error rates $\alpha_1=0.02$ and $\alpha_2 = 0.3$ respectively, the overall FWER is still controlled at $\alpha=0.05$.

Does our conservatism when asking about smoking compensate for our anti-conservatism when asking about coffee? Perhaps not; those readers who are primarily interested in statins will be consistently misled, and readers who are primarily interested in red meat consumption will be see unnecessarily conservative intervals. As such, averaging our error rates across the two questions, with two different interpretations, seems inappropriate.

More problematically, if the different questions correspond to different and non-overlapping models --- for example, if we examine residuals to decide between a Poisson log-linear model and a negative-binomial model --- then it is especially unintuitive to focus on error rates averaged across the different choices of model.

By contrast, if the different questions represent a bag of relatively anonymous, {\em a priori} undifferentiated hypotheses which we are prioritizing for follow-up research, such as in a genome-wise association study, then an error rate like the FDR is likely a better proxy for our scientific goals.

\section{Discussion}\label{sec:discussion}

Selective inference concerns the properties of inference carried out after using a data-dependent procedure to select which questions to ask. We can recover the same long-run frequency properties among answers to {\em selected} questions that we would obtain in the classical non-adaptive setting, if we follow the guiding principle of selective error control:
\begin{center}
The answer must be valid, given that the question was asked.
\end{center}

Happily, living up to this principle can be a simple matter in exponential family models including linear regression, due to the rich classical theory of optimal testing in exponential family models. Even if we are possibly selecting from a large menu of diverse and incompatible models, we can still design tests one model at a time and control the selective error using the test designed for the selected model. We generally pay a price for conditioning, so it is desirable to condition on as little as possible. \capSampOrData carving can dramatically improve on \sampOrData splitting by using the leftover information in $Y_{1}$, the data set initially designated for selection.

Many challenges remain. Deriving the cutoffs for sample carving tests can be computationally difficult in general. In addition, the entire development of this article takes the model selection procedure $\hcQ$ as given, when in reality we can choose $\hcQ$. More work is needed to learn what model selection procedures lead to favorable second-stage properties.

As data sets and research questions become more and more complex, we have less and less hope of specifying  adequate statistical models ahead of time. As such, a key challenge of complex research is to balance the goal of choosing a realistic model against the goal of inference once we have chosen it. We hope that the ideas in this article represent a step in the right direction.

\section*{Reproducibility}

A git repository with code to generate the figures for this file is available at the first author's website.

\section*{Acknowledgements}
William Fithian was supported by National Science Foundation VIGRE grant DMS-0502385 and the Gerald J. Lieberman Fellowship.
Dennis Sun was supported in part by the Stanford Genome Training Program (NIH/NHGRI T32 HG000044) and the Ric Weiland Graduate Fellowship.
Jonathan Taylor was supported in part by National Science Foundation grant DMS-1208857 and
Air Force Office of Sponsored Research grant 113039.
We would like to thank Stefan Wager, Trevor Hastie, Rob Tibshirani, Brad Efron, Yoav Benjamini, Larry Brown, Maxwell Grazier G'sell, Subhabrata Sen, and Yuval Benjamini for helpful discussions.

\bibliographystyle{plainnat}
\bibliography{biblio}

\begin{appendix}

\section{Proof of Proposition~\ref{prop:disciplineWide}}\label{sec:disciplineWideProof}
\begin{proof}
For group $i$, let $R_i$ be the number of true nulls selected, i.e.,
\[R_i = \left|\left\{(M,H_0):\,(M,H_0)\in \hcQ_i(Y_i),\; F_i\in H_0\sub M\right\}\right|,\]
and let $V_i$ denote the number of false rejections. If $Z_n^V=\sum_{i=1}^n V_i$ and $Z_n^R = \sum_{i=1}^n R_i$, then we need to show $\limsup_{n\to\infty} Z_n^V / Z_n^R \leq \alpha$.

By design, ${0 \leq V_i \leq R_i}$ and ${\E(V_i) \leq \alpha\,\E(R_i)}$. As a result, $\E[Z_n^V] / \E[Z_n^R] \leq \alpha$ for every $n$, so we just need to show that the two sums are not far from their expectations. Because
\[
\sum_{i=1}^\infty \frac{\Var(R_i)}{i^2} \leq B\,\sum_{i=1}^\infty \frac{1}{i^2} < \infty,
\]
we can apply Kolmogorov's strong law of large numbers to the independent but non-identical sequence $R_1,R_2,\ldots$ to obtain
\[
\frac{1}{n}(Z_n^R - \E Z_n^R) \toAS 0, \quad\text{ so }\quad \left|\frac{Z_n^R}{\E Z_n^R} - 1\right| \leq \left|\frac{\delta}{n}(Z_n^R - \E Z_n^R)\right| \toAS 0.
\]
As for $Z_n^V$, we have
\[
\frac{1}{n}(Z_n^V - \E Z_n^V) \toAS 0, \quad\text{ so }\quad
\frac{Z_n^V}{\E Z_n^R} - \alpha \leq \frac{\delta}{n}(Z_n^V - \E Z_n^V) \toAS 0;
\]
in other words, $Z_n^R/\E Z_n^R \toAS 1$ and $\limsup_{n} Z_n^V/\E Z_n^R\leqAS \alpha$.
\end{proof}

\section{Monte Carlo Tests and Confidence Intervals: Details}\label{sec:mcUMPU}

Assume $Z$ arises from a one-parameter exponential family
\begin{equation}
  Z \sim g_\theta(z) = e^{\theta z - \psi(\theta)}\,g_0(z).
\end{equation}

We wish to compute (by Monte Carlo) the UMPU two-sided rejection region for the hypothesis $H_0:\,\theta = \theta_0$. Let $U\sim \text{Unif}[0,1]$ be an auxiliary randomization variable.

Define the dictionary ordering on $[0,1]$:
\begin{equation}
  (z_1,u_1) \prec (z_2,u_2) \iff z_1 < z_2 \text{ or } (z_1 = z_2 \text{ and } u_1 < u_2).
\end{equation}
If $\Gamma_1 = (c_1, \gamma_1)$ and $\Gamma_2 = (c_2, 1-\gamma_2)$, then the region
\begin{equation}
  R_{\Gamma_1,\Gamma_2} = \{(z,u):\, (z,u) \prec \Gamma_1 \text{ or } (z,u) \succ \Gamma_2\}
\end{equation}
implements the rejection region for the test with cutoffs $c_1,c_2$ and boundary randomization parameters $\gamma_1,\gamma_2$.

For $\Gamma_1 \prec \Gamma_2$, write
\begin{align}
  K_1(\Gamma_1,\Gamma_2;\theta) &= \P_\theta(R_{\Gamma_1,\Gamma_2}) - \alpha\\
  K_2(\Gamma_1,\Gamma_2;\theta) &= \E_\theta(Z \gv (Z,U)\in R_{\Gamma_1,\Gamma_2}^C) - \E_\theta(Z),
\end{align}
so that the correct cutoffs $\Gamma_i$ are those for which $K_1(\Gamma_1,\Gamma_2;\theta)=K_2(\Gamma_1,\Gamma_2;\theta)=0$.
For fixed $\theta$, $K_1$ is decreasing in $\Gamma_1$ and increasing in $\Gamma_2$, while $K_2$ is increasing in both $\Gamma_1$ and $\Gamma_2$.

Let $(Z_1,W_1), (Z_2,W_2), \ldots$ be a sequence of random variables for which
\begin{align}
  \widehat \E_\theta^n h(Z) &=
  \frac{\sum_{i=1}^n W_i h(Z_i)e^{\theta Z_i} }
  {\sum_{i=1}^n W_ie^{\theta Z_i}}\\
  &\toAS \E_\theta h(Z).
\end{align}
for all integrable $h$. This would be true if $(Z_i,W_i)$ are a valid i.i.d. sample or i.i.d. importance sample from $g_0$, or if they come from a valid Markov Chain Monte Carlo algorithm.

If $\widehat K_i^{n}$ are defined analogously to $K_i$ for $i=1,2$, with $\E_\theta$ and $\P_\theta$ replaced with their importance-weighted empirical versions $\widehat \E_\theta^n$ and $\widehat\P_\theta^n$, then $\widehat K_i^n \toAS K$ pointwise as $n\to\infty$, and  $\widehat K_i^n$ satisfy the same monotonicity properties almost surely for each $n$. As a result, we have almost sure convergence on compacta for $(\widehat K_1^n,\widehat K_2^n)$:
\begin{equation}
  \sup_{(\Gamma_1,\Gamma_2)\in G} \max_i \left\|\widehat K_i^n(\Gamma_1,\Gamma_2;\theta) - K_i(\Gamma_1,\Gamma_2;\theta)\right\|
\end{equation}
for each $\theta$, for compact $G \in \left(\R\times [0,1]\right)^2$.

We carry out our tests by solving for $\Gamma_1$ and $\Gamma_2$ which solve $\widehat K_1^n$ and $\widehat K_2^n$, in effect defining the UMPU tests for a one-parameter exponential family through the approximating empirical measure. Specifically, we can define \begin{equation}
  \widehat \Gamma_2(\Gamma_1;\theta) = \inf \left\{\Gamma_2:\, \widehat K_1^n(\Gamma_1,\Gamma_2;\theta) = 0\right\},
\end{equation}
with $\widehat \Gamma_2=\infty$ if the set is empty. That is, for a given lower cutoff we define the upper cutoff to obtain a level-$\alpha$ acceptance region if that is possible.
Then, $\widehat K_2^n\left(\Gamma_1,\widehat\Gamma_2(\Gamma_1;\theta);\theta\right)$ is an increasing function and we can solve it using binary search. Let $\widehat R_\theta$ denote the rejection region so obtained.

Note that $(z,u)$ is in the left-tail of $\widehat R_\theta$ if and only if ${\widehat K_2^n\left((z,u),\widehat\Gamma_2((z,u));\theta\right) < 0}$. This fact, paired with an analogous test for whether $(z,u)$ is in the right tail, gives us a quick way to carry out the test. It also allows us to quickly find the upper and lower confidence bounds for the approximating empirical family, via binary search.

\section{Sampling for the Selective $t$-Test: Details}\label{sec:sphereSample}

Let $C \sub \R^k$ denote a set with nonempty interior and consider the problem of integrating some integrable function $h(y)$ against the uniform probability measure on $C \cap  S^{k-1}$, where $S^{k-1}$ is the unit sphere of dimension $k-1$, assuming the intersection is non-empty. Assume we are given an i.i.d. sequence of uniform samples $Y_1,Y_2, \ldots$ from $C \cap B^k$, where $B^k$ is the unit ball.

Let $R \sim \frac{r^{k-1}}{k}$, so that if $Z\sim\text{Unif}(S^{k-1})$, then $Y=RZ \sim
\text{Unif}(B^k)$. Let
\begin{equation}
  W(Z) =
  \left(\int_0^1\1\{rZ \in C\}\, \frac{r^{k-1}}{k}\, dr\right)^{-1}
\end{equation}

We can use the $Y_i$ for which $Z_i = Y_i/\|Y_i\| \in C$ as a sequence of importance samples with weights $W(Z_i)$, since
\begin{align}
 \E(h\left(Z\right) &\1\left\{Y,Z\in
      C\right\} W(Z))\\[5pt]
  &= \int_{S^{k-1}}\int_0^1 h(z)\1\left\{z,rz\in C\right\}
  W(z) \frac{r^{k-1}}{k}\, dr\,dz\\[5pt]
  &= \int_{S^{k-1}} h(z)\1\left\{z\in C\right\} \,dz\\[5pt]
  &= \E\left(h(Z) \1\left\{Z\in C\right\}\right).
\end{align}
\newcommand{\bQ}{{Q}}
To carry out the selective $t$-test of $H_0:\; \beta_j = 0$, we need to sample from
\begin{equation}
  \cL\left( \eta'Y \Gv \proj_{\bX_{M\setminus j}} Y,\;\; \|Y\|, \;\;A\right).
\end{equation}

Let $U= \proj_{\bX_{M\setminus j}} Y$, and let $\bQ\in\R^{n\times (n-|M|-1)}$ be such that $\bQ\bQ'=\proj_{\bX_{M\setminus j}}^\perp$. Then $L^2\triangleq \|\bQ'Y\|^2 = \|Y\|^2 - \|U\|^2$ is fixed under the selection event. Let
\begin{equation}
  C = \{v:\, U+\bQ v\in A\},
\end{equation}
so that $A_U = U + \bQ C$, an $(n-|M|-1)$-dimensional hyperplane intersected with $A$, is the event we would sample from for the selective $z$-test.

Under $H_0$, $Y$ is uniformly distributed on \begin{equation}
  (U+\bQ C)\cap \|Y\| S^{n-1} = U+\bQ \left(C \cap L S^{n-|M|-2}\right).
\end{equation}

Assume we can resample $Y^*$ uniformly from $A_U \cap (U + LB^{n-|M|-1})$, which is just sampling from $A_U$ with an additional quadratic constraint. Then $V^* = \bQ' (Y^*-U)$ is a sample from the ball of radius $L$, intersected with $C$. We can turn $V^*$ into an importance-weighted sample from the sphere via the scheme outlined above; then, the same importance weight suffices to turn $Y^*$ into a sample from the selective $t$-test conditioning set.

\end{appendix}

\end{document}